\documentclass[11 pt]{amsart}
\usepackage{amscd,amsfonts,amssymb,amsmath}
\usepackage{hyperref}
\usepackage{epsfig}
\usepackage{mathtools}
\usepackage{tabu}
\usepackage{tikz-cd}
\newtheorem{theorem}{Theorem}[section]
\newtheorem{corollary}[theorem]{Corollary}
\newtheorem{lemma}[theorem]{Lemma}
\newtheorem{proposition}[theorem]{Proposition}
\theoremstyle{definition}

\newtheorem{problem}[theorem]{Problem}

\newtheorem{example}[theorem]{Example}
\newtheorem{remark}[theorem]{Remark}
\numberwithin{equation}{subsection}

\usepackage[all,cmtip]{xy}

\usepackage{graphicx} 

\newcommand{\Alex}{\operatorname{Alex}}
\newcommand{\Maps}{\operatorname{Map}}
\newcommand{\Aut}{\operatorname{Aut}}

\newcommand{\Conj}{\operatorname{Conj}}
\newcommand{\Core}{\operatorname{Core}}

\newcommand{\Ker}{\operatorname{Ker}}

\newcommand{\Adj}{\operatorname{Adj}}

\newcommand{\Inn}{\operatorname{Inn}}

\newcommand{\T}{\operatorname{T}}
\newcommand{\Ho}{\operatorname{H}}

\newcommand{\B}{\operatorname{B}}
\newcommand{\R}{\operatorname{R}}
\newcommand{\Hom}{\operatorname{Hom}}

\newcommand{\Z}{\operatorname{Z}}
\newcommand{\img}{{\rm Im}}
\newcommand{\Map}{\mathrm{Map}}
\newcommand{\id}{\mathrm{id}}

\begin{document}

\title{Quandle cohomology, extensions and automorphisms}

\author{Valeriy Bardakov}
\address{Sobolev Institute of Mathematics and Novosibirsk State University, Novosibirsk 630090, Russia.
Novosibirsk State Agrarian University, Dobrolyubova street, 160, Novosibirsk, 630039, Russia.}
\email{bardakov@math.nsc.ru}
\author{Mahender Singh}
\address{Department of Mathematical Sciences, Indian Institute of Science Education and Research (IISER) Mohali, Sector 81,  S. A. S. Nagar, P. O. Manauli, Punjab 140306, India.}
\email{mahender@iisermohali.ac.in}


\begin{abstract}
A quandle is an algebraic system with a binary operation satisfying three axioms modelled on the three Reidemeister moves of planar diagrams of links in the 3-space. The paper establishes new relationship between cohomology, extensions  and automorphisms of quandles. We derive a four term exact sequence relating quandle 1-cocycles, second quandle cohomology and certain group of automorphisms of an abelian extension of quandles. A non-abelian counterpart of this sequence involving dynamical cohomology classes is also established, and some applications to lifting of quandle automorphisms are given.  Viewing the construction of the conjugation, the core and the generalised Alexander quandle  of a group as an adjoint functor of some appropriate functor from the category of quandles to the category of groups, we prove that these functors map extensions of groups to extensions of quandles. Finally, we construct some natural group homomorphisms from the second cohomology of a group to the second cohomology of its core and conjugation quandles.
\end{abstract}
\maketitle

\section{Introduction}\label{introduction}

A quandle is an algebraic system with a binary operation satisfying three axioms that are algebraic formulations of the three Reidemeister moves of planar diagrams of links in the 3-space. These objects have shown appearance in a wide spectrum of mathematics including knot theory \cite{Joyce1982, Matveev}, group theory, set-theoretic solutions to the quantum Yang-Baxter equation and Yetter-Drinfeld Modules \cite{Eisermann2005}, Riemannian symmetric spaces \cite{Loos1} and Hopf algebras \cite{Andruskiewitsch2003}, to name a few. Though already studied under different guises in the literature, study of these objects gained momentum after the fundamental works of Matveev \cite{Matveev} and Joyce \cite{Joyce1982}, who showed that link quandles are complete invariants of non-split links up to orientation of the ambient space. Although link quandles are strong invariants, it is difficult to check whether two quandles are isomorphic. This motivated search for newer properties and invariants of quandles themselves. The curious reader may refer to the articles \cite{Carter2012, Kamada, FennRourke1992, Nelson2011, Nosaka2017} for more on the historical development of the subject.
\par

Recent years have seen many interesting works related to quandles and racks, their weaker analogues.  A (co)homology theory for racks was introduced in \cite{FennRourkeSanderson1995, FennRourkeSanderson} as homology of the rack space of a rack. This motivated development of a (co)homology theory for quandles in \cite{CJKLS2003} where state-sum invariants using quandle cocycles as weights were defined.  These theories have been further developed and generalized in \cite{Andruskiewitsch2003, Lebed2013} which encompasses all the previously known theories. A recent work \cite{Szymik2019} has shown that quandle cohomology is a Quillen cohomology which is the cohomology group of a functor from the category of models to that of complexes. In the cohomology theory of groups, it is well-known that 2-cocycles are closely related to extensions of groups \cite{Brown1981}. Analogous theories have been developed for quandle cocycles \cite{CES02}. Besides being of independent algebraic interest these are useful in computing cocycle knot invariants. Abelian extensions of quandles have been investigated in \cite{CKS03} using diagrammatic approach. Further generalizations of quandle extensions by dynamical cocycles have been given in \cite{Andruskiewitsch2003}, and a more general abelian extension theory for racks and quandles, which also generalizes other variants in the literature, has been detailed in  \cite{Jackson2005}. Using the generalized quandle cohomology theory of \cite{Andruskiewitsch2003} some new knot invariants have been introduced in \cite{Carter2005}. Quandle cocycle invariants have been computed from quandle coloring numbers in \cite{ClarkSaitoVendramin2016}. Some properties of quandle extensions and of invariants based on 2-cocycles are investigated in \cite{ClarkSaito2016}. For example, it is proved that if the quandle extension is a conjugation quandle, then the state-sum invariant given by the 2-cocycle is constant. An explicit description of the second quandle cohomology of a finite indecomposable quandle has been given in \cite{GarciaVendramin2017} developing further ideas from \cite{EtingofGrana2003} that related the second cohomology of a quandle to the first cohomology of its enveloping group.
\par

Automorphisms of quandles reveal a lot about their internal structures and have been investigated in much detail in a series of works \cite{Bardakov2017, BarTimSin, Elhamdadi2012, FermanNowikTeicher2011, Hou2011}. Study of automorphisms of quandles is important from the point of view of knot theory. The totality of the number of colorings of diagram of a knot $K$ by a quandle $X$ is in bijection with the set $\Hom(Q(K), X)$ of homomorphisms from the knot quandle $Q(K)$ to $X$. Composing a fixed coloring of the knot diagram by automorphisms of $X$ gives rise to new colorings of the knot diagram. A recent work \cite{Horvat2017} has explored further relationship between the knot symmetries and the automorphisms of the knot quandle. 
\par

The purpose of this paper is to establish some interestingly new connections between cohomology, extensions and automorphisms of quandles. We also relate these objects to cohomology and extensions of groups.  
\par

The paper is organised as follows. Section \ref{sec-prelim} recalls some basic definitions that are used in the paper. In Section \ref{Quandle-cocycles-extensions}, we give a generalisation of the construction of a quandle extension due to \cite{Andruskiewitsch2003} (Proposition \ref{set-cocycle}).  In Section \ref{Action of automorphisms on dynamical 2-cocycles}, we prove that if $X$ is a quandle, $\Sigma_S$ the symmetric group on a set $S$ and $\mathcal{H}^2(X; S)$ the set of cohomology classes of dynamical 2-cocycles (in the sense of \cite{Andruskiewitsch2003}), then there is an action of the group $\Aut(X) \times \Sigma_S$ on $\mathcal{H}^2(X; S)$ (Proposition \ref{action-auto-on-dynamical-classes}).  This action plays a crucial role in the main results of the paper. In Section \ref{Exact sequence automorphisms dynamical cohomology}, we derive a four term exact sequence relating certain group of automorphisms of a quandle extension of $X$ by $S$ and $\mathcal{H}^2(X; S)$. More precisely, we prove in Theorem \ref{main-thm-1} that if $x_0 \in X$ is a fixed base point and $E=X \times_\alpha S$ an extension of $X$ by $S$ using a dynamical 2-cocycle $\alpha$, then there exists an exact sequence
\begin{equation}\label{intro-equation-1}
1 \longrightarrow \Ker(\Phi) \longrightarrow \Aut^{x_0}_S(E) \stackrel{\Phi}{\longrightarrow} \Aut^{x_0}(X) \times \Sigma_S \stackrel{\Omega}{\longrightarrow} \mathcal{H}^2(X; S),
\end{equation}
where $\Aut^{x_0}_S(E)$ is certain subgroup of $\Aut(E)$. Section \ref{abelian-quandle-cohomology-section} proves an abelian counterpart of the sequence \eqref{intro-equation-1} in which case the sequence takes a simpler and elegant form. We prove in Theorem \ref{abelian-main-theorem} that if $A$ is an abelian group and $E=X \times_\alpha A$ an abelian extension of $X$ by $A$ using a quandle 2-cocycle $\alpha$, then there exists an exact sequence
\begin{equation}\label{intro-equation-2}
1 \longrightarrow \Z^1(X;A) \longrightarrow \Aut_A(E) \stackrel{\Psi}{\longrightarrow} \Aut(X) \times \Aut(A) \stackrel{\Theta}{\longrightarrow} \Ho^2(X; A),
\end{equation}
where $\Aut_A(E)$ is a subgroup of $\Aut(E)$, $ \Z^1(X;A)$ is the group of quandle 1-cocycles and $\Ho^2(X; A)$ is the usual second quandle cohomology of $X$ with coefficients in $A$. The maps  $\Omega$ and $\Theta$ are not group homomorphisms in general. In Theorem \ref{main-thm-3} of Section \ref{properties of map theta}, we prove that  $\Theta$ is a derivation
with respect to the action of $\Aut(X) \times \Aut(A)$ on $\Ho^2(X; A)$.  Further, if $\Theta$ and $\Theta'$ correspond to two  quandle 2-cocycles, then $\Theta$ and $\Theta'$ differ by an inner derivation. Section \ref{adjoint-functor-section} shows that the construction of the conjugation, the core and the generalised Alexander quandle of a group can be viewed as an adjoint functor of some appropriate functor from the category of quandles to the category of groups (Propositions \ref{verbal-adjoint} and \ref{alexander-adjoint}). In Section \ref{groups-to-quandles-extensions-section}, we show that  these functors map extensions of groups to extensions of quandles (Propositions \ref{core-conj-prop} and \ref{Alex-prop}).  Finally, in Section \ref{Homomorphism group cohomology to quandle cohomology}, we construct some natural group homomorphisms from the second group cohomology of a group to the second quandle cohomology of its core and conjugation quandles. More precisely, Theorems \ref{homo-group-coho-quandle-coho} and \ref{group-homo-grouo-coho-conj} prove that if $G$ is a group and $A$ a trivial $G$-module, then there are natural group homomorphisms $$\Lambda:\Ho^2(G;A)_{sym} \to \mathcal{H}^2 \big(\Core(G);A\big )$$ and $$\Gamma:\Ho^2(G;A) \to \Ho^2\big(\Conj(G);A\big),$$ 
where $\Ho^2(G;A)_{sym}$ is the group of symmetric cohomology classes that correspond to extensions of $G$ by $A$ in which all groups are abelian.
\par

Throughout the paper, the results are stated for quandles only but they hold for racks as well. For notational convenience, sometimes we denote the value of a map $\phi$ at a point $x$ by $\phi_x$. All groups are written multiplicatively.
\medskip

\section{Preliminaries on quandles}\label{sec-prelim}
This section recalls some basic definitions for the convenience of the reader.
\par

A {\it quandle} is a non-empty set $X$ with a binary operation $(x,y) \mapsto x * y$ satisfying the following axioms:
\begin{enumerate}
\item[(Q1)] $x*x=x$ for all $x \in X$,
\item[(Q2)] For any $x,y \in X$ there exists a unique $z \in X$ such that $x=z*y$,
\item[(Q3)] $(x*y)*z=(x*z) * (y*z)$ for all $x,y,z \in X$.
\end{enumerate}

An algebraic system satisfying only (Q2) and (Q3)  is called a {\it rack}. Many interesting examples of quandles come from groups showing deep connections with group theory.

\begin{itemize}
\item If $G$ is a group and $n$ an integer, then the set $G$ equipped with the binary operation $a*b= b^{-n} a b^n$ forms a quandle $\Conj_n(G)$. For $n=1$, it is the {\it conjugation quandle} $\Conj(G)$.
\item If $f \in \Aut(G)$ is an automorphism of a group $G$, then the set $G$ with the binary operation $x*y=f(xy^{-1})y$ forms a quandle $\Alex_f (G)$ called the \textit{generalised  Alexander quandle} of $G$. In particular, if $G= \mathbb{Z}/n \mathbb{Z}$ and $\phi$ is the inversion, then we get the {\it dihedral quandle} $\R_n$.
\item If $G$ is a group, then the binary operation $a*b= b a^{-1} b$ turns the set $G$ into the {\it core quandle}  $\Core(G)$.
\end{itemize}
\medskip

A quandle  $X$ is called {\it trivial} if $x*y=x$ for all $x, y \in X$.  Unlike groups, a trivial quandle can contain arbitrary number of elements. Notice that the axioms (Q2) and (Q3) are equivalent to saying that the map $S_x: X \to X$ given by $$S_x(y)=y*x$$ is an automorphism of $X$ for each $x \in X$. Each such $S_x$ is called the {\it inner automorphism} of $X$ induced by $x$, and the group $\Inn(X)$ generated by such automorphisms is called the group of inner automorphisms. The group of all automorphisms of $X$ is denoted by $\Aut(X)$. A quandle $X$ is {\it connected} if $\Inn(X)$ acts transitively on $X$. Every quandle $X$ gives rise to a universal group $$\Adj(X)= \big\{e_x,~x \in X~|~e_{x*y}=e_ye_xe_y^{-1}, x, y \in X \big\},$$ called the {\it adjoint} or {\it associated} or {\it enveloping group} of  $X$. If $\mathcal{Q}$ is the category of quandles and $\mathcal{G}$ the category of groups, then the functor
$\Adj: \mathcal{Q} \to \mathcal{G}$ is left adjoint to the functor $\Conj: \mathcal{G} \to \mathcal{Q}$ and the adjoint group of a knot quandle is the knot group \cite[Section 15]{Joyce1982}. 
\medskip

\section{Quandle cocycles and extensions}\label{Quandle-cocycles-extensions}

Extension theory for quandles using quandle cocycles was first proposed in \cite{CES02, CKS03}. A far reaching generalization of quandle extensions using dynamical 2-cocycles which encompasses all the previous constructions was given in \cite{Andruskiewitsch2003}. We begin with a generalisation of a construction of Andruskiewitsch and Grana in \cite[Lemma 2.1]{Andruskiewitsch2003}.

\begin{proposition}\label{set-cocycle}
Let $X$ and $S$ be two sets, $\alpha: X \times X \to \Maps(S \times S, S)$ and $\beta: S \times S \to \Maps(X \times X, X)$ two maps. Then the set $X \times S$ with the binary operation 
\begin{equation}\label{genralised-quandle-operation}
(x, s)* (y,t)= \big( \beta_{s, t}(x, y), ~\alpha_{x, y}(s, t) \big)
\end{equation}
forms a quandle if and only if the following conditions hold:
\begin{enumerate}
\item $\beta_{s, s}(x, x)=x$ and $\alpha_{x, x}(s, s)=s$ for all $x \in X$, $s \in S$;
\item  for each $(y, t) \in X \times S$, the map $(x, s) \mapsto  \big( \beta_{s, t}(x, y),~ \alpha_{x, y}(s, t) \big)$ is a bijection;
\item $\beta_{\alpha_{x, y}(s, t), u}\Big(\beta_{s, t}(x, y),~ z \Big)= \beta_{\alpha_{x, z}(s, u), \alpha_{y, z}(t, u)}\Big(\beta_{s, u}(x, z), ~\beta_{t, u}(y, z) \Big)$ 
\item[] and
\item[] $\alpha_{\beta_{s, t}(x, y), z}\Big(\alpha_{x, y}(s, t), ~u \Big)= \alpha_{\beta_{s, u}(x, z), \beta_{t, u}(y, z)}\Big(\alpha_{x, z}(s, u), ~\alpha_{y, z}(t, u) \Big)$.
\end{enumerate}
\end{proposition}

\begin{proof}
The axiom (Q1) holds if and only if $\beta_{s, s}(x, x)=x$ and $\alpha_{x, x}(s, s)=s$ for all $x \in X$, $s \in S$. The axiom (Q2) is equivalent to condition (2). Finally, if $(x, s), (y, t), (z, u) \in X \times S$, then 
\begin{eqnarray*}
\big((x, s)*(y, t) \big)* (z, u) &=& \big( \beta_{s, t}(x, y), ~\alpha_{x, y}(s, t) \big) *(z, u)\\
&=& \Big(\beta_{\alpha_{x, y}(s, t), u}\big(\beta_{s, t}(x, y), z \big),~ \alpha_{\beta_{s, t}(x, y), z}\big(\alpha_{x, y}(s, t), u \big) \Big)
\end{eqnarray*}
and
\begin{eqnarray*}
&& \big((x, s)*(z, u) \big)* \big((y, t)*(z, u) \big)\\
&=& \big( \beta_{s, u}(x, z), ~\alpha_{x, z}(s, u) \big) *\big( \beta_{t, u}(y, z), ~\alpha_{y, z}(t, u) \big)\\
&=& \Big(\beta_{\alpha_{x, z}(s, u), ~\alpha_{y, z}(t, u)}\big(\beta_{s, u}(x, z), ~\beta_{t, u}(y, z) \big), ~\alpha_{\beta_{s, u}(x, z), ~\beta_{t, u}(y, z)}\big(\alpha_{x, z}(s, u), ~\alpha_{y, z}(t, u) \big)\Big).
\end{eqnarray*}
Thus, axiom (Q3) is equivalent to condition (3).
\end{proof}

We denote the quandle obtained in Proposition \ref{set-cocycle} by $X~ _{\alpha}\times_{\beta} S$. 

\begin{corollary}
Let $X$ and $S$ be two sets, $\alpha: X \times X \to \Maps(S \times S, S)$ and $\beta: S \to \Maps(X \times X, X)$ two maps. Then the set $X \times S$ with the binary operation 
$$(x, s)* (y,t)= \big( \beta_{t}(x, y), \alpha_{x, y}(s, t) \big)$$ forms a quandle if and only if the following conditions hold:
\begin{enumerate}
\item $\beta_{s}(x, x)=x$ and $\alpha_{x, x}(s, s)=s$ for all $x \in X$, $s \in S$;
\item $\beta_{s}(-, y): X \to X$ and  $\alpha_{x, y}(-, t): S \to S$ are bijections for all $x, y \in X$, $s \in S$;
\item $\beta_{u}\big(\beta_{t}(x, y), ~z \big)= \beta_{\alpha_{y, z}(t, u)}\big(\beta_{u}(x, z), ~\beta_{u}(y, z) \big)$ 
\item[] and
\item[] $\alpha_{\beta_{t}(x, y), z}\big(\alpha_{x, y}(s, t), ~u \big)= \alpha_{\beta_{u}(x, z), ~\beta_{u}(y, z)}\big(\alpha_{x, z}(s, u), ~\alpha_{y, z}(t, u) \big)$.
\end{enumerate}
\end{corollary}
\medskip

Andruskiewitsch and Grana \cite[Section 2.1]{Andruskiewitsch2003} considered a special case of Proposition \ref{set-cocycle} when $X$ is a quandle and $$\beta_{s, t}(x, y)=x*y$$ for all $x, y \in X$ and $s, t \in S$. In this case the map $\alpha$ satisfy the
conditions
\begin{equation}\label{dynamical-cocycle-condition1}
\alpha_{x, x}(s, s)=s,
\end{equation}
\begin{equation}\label{dynamical-cocycle-condition2}
\alpha_{x, y}(-, t): S \to S~\textrm{is a bijection}
\end{equation}
and the  {\it cocycle condition}
\begin{equation}\label{dynamical-cocycle-condition3}
\alpha_{x*y, z}\big(\alpha_{x, y}(s, t), ~u \big)= \alpha_{x* z, y*z}\big(\alpha_{x, z}(s, u),~ \alpha_{y, z}(t, u) \big)
\end{equation}
for all $x, y \in X$ and $s, t \in S$. Such an $\alpha$ is referred as a {\it dynamical 2-cocycle}. Further, in this case, the quandle operation \eqref{genralised-quandle-operation}  on $X \times S$ becomes 
\begin{equation}\label{dynamical-quandle-operation}
(x, s)* (y,t)= \big( x* y, ~\alpha_{x, y}(s, t) \big).
\end{equation} 
The quandle so obtained is called an  {\it extension} of $X$ by $S$ through $\alpha$ and is denoted by $X \times_{\alpha} S$. It is easy to see that for each $x \in X$, the set $S$ forms a quandle with the binary operation $$s *_x t :=\alpha_{x,x}(s,t)$$ for all $s, t \in S$.

\begin{corollary}\label{isomorphic-s-quandles}
If $X$ is a connected quandle, then $(S, *_x) \cong (S, *_y)$ for all $x, y \in X$.
\end{corollary}

\begin{proof}
Let $x, z \in X$ and $u \in S$. Define a map $\phi: (S, *_x) \to (S, *_{x*z})$ by 
$$\phi(s)= \alpha_{x, z}(s, u)$$ for all $s \in S$.
Then, for $s, t \in S$, we have
\begin{eqnarray*}
\phi(s *_x t) &=& \alpha_{x, z}(s *_x t, u)\\
&=& \alpha_{x, z}\big( \alpha_{x, x}(s, t), u\big)\\
&=&\alpha_{x* z, x*z}\big(\alpha_{x, z}(s, u),~ \alpha_{x, z}(t, u) \big),~\textrm{by \eqref{dynamical-cocycle-condition3}}\\
&=&\phi(s) *_{x*z} \phi(t),
\end{eqnarray*}
and $\phi$ is an isomorphism of quandles. Since $X$ is connected, it follows that $(S, *_x) \cong (S, *_y)$ for all $x, y \in X$.
\end{proof}

\section{Action of automorphisms on dynamical 2-cocycles}\label{Action of automorphisms on dynamical 2-cocycles}

Let $X$ be a quandle, $\Sigma_S$ the symmetric group on a set $S$ and $\mathcal{Z}^2(X; S)$ the set of all dynamical 2-cocycles. For $(\phi, \theta) \in \Aut(X) \times \Sigma_S$ and a dynamical 2-cocycle $\alpha \in \mathcal{Z}^2(X; S)$, define $$^{(\phi, \theta)}\alpha:X \times X \to \Map(S \times S, S)$$ by setting
\begin{equation}\label{auto-action}
^{(\phi, \theta)}\alpha_{x, y}(s, t):= \theta \Big(\alpha_{\phi^{-1}(x), \phi^{-1}(y)}\big(\theta^{-1}(s), ~\theta^{-1}(t)\big)  \Big)
\end{equation}
for $x, y \in X$ and $s, t \in S$.

\begin{lemma}\label{aut-action-cocycles}
The map $\big(\Aut(X) \times \Sigma_S\big) \times \mathcal{Z}^2(X; S) \to \mathcal{Z}^2(X; S)$ given by
$$\big(( \phi, \theta), \alpha\big) \mapsto~ ^{(\phi, \theta)}\alpha$$ defines a left action of the group $\Aut(X) \times \Sigma_S$ on the set $\mathcal{Z}^2(X; S)$.
\end{lemma}
\begin{proof} 
Let $(\phi, \theta) \in \Aut(X) \times \Sigma_S$ and $\alpha \in \mathcal{Z}^2(X; S)$. We first show that $^{(\phi, \theta)}\alpha$ is a dynamical 2-cocycle. For $x, y \in X$ and $s, t \in S$, we compute
\begin{eqnarray*}
&  & ^{(\phi, \theta)}\alpha_{x*y, z}\Big(^{(\phi, \theta)}\alpha_{x, y}(s, t), u \Big) \\
& = & \theta \Big(\alpha_{\phi^{-1}(x)*\phi^{-1}(y), ~\phi^{-1}(z)}\Big(\theta^{-1}\big(^{(\phi, \theta)}\alpha_{x, y}(s, t)\big),~ \theta^{-1}(u)\Big)  \Big)\\
& = & \theta \Big(\alpha_{\phi^{-1}(x)*\phi^{-1}(y), ~\phi^{-1}(z)}\Big(\alpha_{\phi^{-1}(x), ~\phi^{-1}(y)}\big(\theta^{-1}(s),~ \theta^{-1}(t)\big),~ \theta^{-1}(u)\Big)  \Big)\\
& = & \theta \Big( \alpha_{\phi^{-1}(x)* \phi^{-1}(z),~ \phi^{-1}(y)*\phi^{-1}(z)}\Big(\alpha_{\phi^{-1}(x),~ \phi^{-1}(z)} \big(\theta^{-1}(s), ~\theta^{-1}(u)\big), ~\alpha_{\phi^{-1}(y), ~\phi^{-1}(z)}\big(\theta^{-1}(t), ~\theta^{-1}(u)\big) \Big) \Big)\\
& & \textrm{by the cocycle condition \eqref{dynamical-cocycle-condition3} for}~\alpha \\
& = & \theta \Big( \alpha_{\phi^{-1}(x*z), \phi^{-1}(y*z)}\Big(\alpha_{\phi^{-1}(x), ~\phi^{-1}(z)}\big(\theta^{-1}(s),~ \theta^{-1}(u)\big),~ \alpha_{\phi^{-1}(y), ~\phi^{-1}(z)}\big(\theta^{-1}(t),~ \theta^{-1}(u)\big) \Big) \Big)\\
& = & ^{(\phi, \theta)}\alpha_{x* z, y*z}\Big(^{(\phi, \theta)}\alpha_{x, z}(s, u),~~ ^{(\phi, \theta)}\alpha_{y, z}(t, u) \Big),
\end{eqnarray*}
and hence $^{(\phi, \theta)}\alpha$ satisfy \eqref{dynamical-cocycle-condition3}. It follows easily that $ ^{(\phi, \theta)}\alpha$ satisfy \eqref{dynamical-cocycle-condition1} and \eqref{dynamical-cocycle-condition2} which shows that $ ^{(\phi, \theta)}\alpha \in \mathcal{Z}^2(X; S)$.
\par

Given $(\phi_i, \theta_i) \in \Aut(X) \times \Sigma_S$ for $i=1, 2$, we see that
\begin{eqnarray*}
^{(\phi_1, \theta_1)(\phi_2, \theta_2)}\alpha_{x, y}(s, t) & = & ^{(\phi_1\phi_2, ~\theta_1\theta_2)}\alpha_{x, y}(s, t) \\
& = & \theta_1\theta_2 \Big(\alpha_{\phi_2^{-1}\phi_1^{-1}(x), ~\phi_2^{-1}\phi_1^{-1}(y)}\big(\theta_2^{-1}\theta_1^{-1}(s), ~\theta_2^{-1}\theta_1^{-1}(t)\big)  \Big)\\
& = & \theta_1 \Big( ^{(\phi_2, \theta_2)}\alpha_{\phi_1^{-1}(x), ~\phi_1^{-1}(y)}\big(\theta_1^{-1}(s),~ \theta_1^{-1}(t)\big) \Big)\\
& = & ^{(\phi_1, \theta_1)}\big(^{(\phi_2, \theta_2)}\alpha\big)_{x, y}(s, t),
\end{eqnarray*}
and hence $\Aut(X) \times \Sigma_S$ acts on $\mathcal{Z}^2(X; S)$ from the left. 
\end{proof}

Two dynamical 2-cocycles $\alpha, \beta \in \mathcal{Z}^2(X; S)$ are said to be {\it cohomologous } if there exists a map $\lambda:X \to \Sigma_S$ such that
\begin{equation}\label{cohomologous}
\beta_{x, y}(s, t)=\lambda_{x*y} \Big(\alpha_{x, y}\big(\lambda_x^{-1}(s),~ \lambda_y^{-1}(t) \big)\Big)
\end{equation}
for all $x, y \in X$ and $s, t \in S$. It can be seen easily that being cohomologous is an equivalence relation on $\mathcal{Z}^2(X; S)$, and we denote the set of cohomology classes of dynamical 2-cocycles by $\mathcal{H}^2(X; S)$. Further, we denote the equivalence class of $\alpha \in \mathcal{Z}^2(X; S)$ by $[\alpha]  \in \mathcal{H}^2(X; S)$.

\begin{proposition}\label{action-auto-on-dynamical-classes}
The map $\big(\Aut(X) \times \Sigma_S\big) \times \mathcal{H}^2(X; S) \to \mathcal{H}^2(X; S)$ given by
$$\big(( \phi, \theta), [\alpha] \big) \mapsto~ [^{(\phi, \theta)}\alpha]$$ defines a left action of the group $\Aut(X) \times \Sigma_S$ on the set $\mathcal{H}^2(X; S)$.
\end{proposition}

\begin{proof}
It only needs to be shown that cohomologous cocycles are mapped to cohomologous cocycles under the action in Lemma \ref{aut-action-cocycles}. Let $(\phi, \theta) \in \Aut(X) \times \Sigma_S$,  $\alpha, \beta \in \mathcal{Z}^2(X; S)$ and $\lambda:X \to \Sigma_S$ such that \eqref{cohomologous} is satisfied. Define a map $\lambda':X \to \Sigma_S$ by setting $\lambda'_x= \theta \lambda_{\phi^{-1}(x)} \theta^{-1}$ for $x \in X$. Then, for $x, y \in X$ and $s, t \in S$, we have
\begin{eqnarray*}
&  & ^{(\phi, \theta)}\beta_{x, y}(s, t) \\
& = & \theta \Big(\beta_{\phi^{-1}(x),~ \phi^{-1}(y)}\big(\theta^{-1}(s), ~~\theta^{-1}(t)\big)  \Big)\\ 
& = & \theta \Big(\lambda_{\phi^{-1}(x)* \phi^{-1}(y)} \Big(\alpha_{\phi^{-1}(x), ~\phi^{-1}(y)} \Big( \lambda_{\phi^{-1}(x)}^{-1}\big(\theta^{-1}(s) \big), ~~\lambda_{\phi^{-1}(y)}^{-1}\big(\theta^{-1}(t) \big) \Big) \Big) \Big),\\ 
& & \textrm{by using}~ \eqref{cohomologous}\\
& = & \theta \Big(\lambda_{\phi^{-1}(x*y)} \Big(\alpha_{\phi^{-1}(x),~ \phi^{-1}(y)} \Big( \lambda_{\phi^{-1}(x)}^{-1}\theta^{-1}(s), ~~\lambda_{\phi^{-1}(y)}^{-1}\theta^{-1}(t) \Big) \Big) \Big)\\
& = & \Big( \theta \lambda_{\phi^{-1}(x*y)} \theta^{-1} \Big) \theta \Big(\alpha_{\phi^{-1}(x),~ \phi^{-1}(y)} \Big( \theta^{-1} (\theta \lambda_{\phi^{-1}(x)}\theta^{-1})^{-1}(s), ~~\theta^{-1}(\theta \lambda_{\phi^{-1}(y)}\theta^{-1})^{-1}(t)  \Big) \Big)\\  
& = & \lambda'_{x*y} ~ \theta \Big(\alpha_{\phi^{-1}(x), ~\phi^{-1}(y)} \Big( \theta^{-1} {\lambda'_x}^{-1}(s), ~~\theta^{-1}{\lambda'}_y^{-1}(t)  \Big) \Big)\\  
& = & \lambda'_{x*y}  ~ \Big(^{(\phi, \theta)}\alpha_{x, y} \Big({\lambda'_x}^{-1}(s), ~~{\lambda'_y}^{-1}(t)  \Big) \Big),  
\end{eqnarray*}
and hence $[^{(\phi, \theta)}\alpha]=[^{(\phi, \theta)}\beta]$.
\end{proof}

\begin{remark}\label{abelian-remark}
Let $X$ be a quandle, $S=A$ a group (not necessarily abelian) and $\alpha: X \times X \to A$ a map. Then the map $\alpha': X \times X \to \Map(A \times A, A)$ given by $$\alpha'_{x, y}(s, t)= s ~\alpha_{x, y},$$ where $x, y \in X$ and $s, t \in A$, satisfy \eqref{dynamical-cocycle-condition1} and \eqref{dynamical-cocycle-condition3} if and only if $\alpha$ satisfy
\begin{equation}\label{group-coefficient-cocycle-condition}
\alpha_{x, y}~\alpha_{x*y, z}= \alpha_{x, z}~\alpha_{x* z, y*z}
\end{equation}
and 
\begin{equation}\label{normalised-cocycle-condition}
\alpha_{x, x}=1
\end{equation}
for $x, y, z \in X$. If $A$ is abelian, then \eqref{group-coefficient-cocycle-condition} is the cocycle condition defining the usual second quandle cohomology $\Ho^2(X, A)$ \cite{CJKLS2003}, which will be discussed in some detail in Section \ref{abelian-quandle-cohomology-section}.
\end{remark}
\medskip

\section{Exact sequence relating automorphisms and dynamical cohomology}\label{Exact sequence automorphisms dynamical cohomology}
Let $X$ be a quandle, $S$ a set, $\alpha\in \mathcal{Z}^2(X; S)$ and $E=X \times_\alpha S$ the extension of $X$ by $S$ using $\alpha$. For a fixed base point $x_0 \in X$, we set
$$\Aut^{x_0}(X)= \big\{\phi \in \Aut(X)~|~ \phi(x_0)=x_0 \big\},$$ 
and 
\begin{small}
$$\Aut^{x_0}_S(E)=\Big\{ \psi \in \Aut(E)~|~\psi(x, s)=\big(\phi(x),~ \tau(x, s)\big)~\textrm{for some}~\phi \in \Aut^{x_0}(X)~\textrm{and}~\tau \in \Map(X \times S, S) \Big\}.$$
\end{small}

We write $\tau(x, s)$ as $\tau_x(s)$ for $x \in X$ and $s \in S$.

\begin{lemma}\label{defining-theta}
The following hold:
\begin{enumerate}
\item If $\psi \in \Aut(E)$ such that $\psi(x, s)=\big(\phi(x), \tau_x(s)\big)$, then $\tau_x \in \Sigma_S$.
\item $\Aut^{x_0}_S(E)$ is a subgroup of $\Aut(E)$.
\end{enumerate}
\end{lemma}

\begin{proof}
Assertion (1) follows from bijectivity of $\psi$. For the second assertion, let $\psi \in \Aut^{x_0}_S(E)$ such that $\psi(x, s)=\big(\phi(x), \tau_x(s)\big)$ for all $x \in X$ and $s \in S$. Define $\varphi(x,s) := \big(\phi^{-1}(x), ~\tau_{\phi^{-1}(x)}^{-1}(s)\big)$. Since $\phi$ and $\tau_x$ are bijections for all $x \in X$, it follows that $\varphi$ is also a bijection. Since $\psi$ is a quandle homomorphism, for $x, y \in X$ and $s, t \in S$, we have
\begin{equation*}
\tau_{x*y}\big(\alpha_{x, y}(s, t)\big)=\alpha_{\phi(x), \phi(y)}\big(\tau_x(s),~ \tau_y(t)\big),
\end{equation*}
which is equivalent to
\begin{equation}\label{useful-in-subgroup}
\alpha_{\phi^{-1}(x), \phi^{-1}(y)} \big(\tau_{\phi^{-1}(x)}^{-1}(s),~ \tau_{\phi^{-1}(y)}^{-1}(t) \big)=\tau_{\phi^{-1}(x*y)}^{-1} \big(\alpha_{x, y}(s,t) \big).
\end{equation}

Now, consider
\begin{eqnarray*}
\varphi \big( (x, s)*(y, t) \big) &=& \varphi \big(x*y,~ \alpha_{x, y}(s, t) \big)\\
 &=&  \big(\phi^{-1}(x*y), ~\tau_{\phi^{-1}(x*y)}^{-1}\big(\alpha_{x, y}(s, t) \big)\big)\\
  &=&  \big(\phi^{-1}(x)*\phi^{-1}(y), ~\alpha_{\phi^{-1}(x), \phi^{-1}(y)} \big(\tau_{\phi^{-1}(x)}^{-1}(s),~ \tau_{\phi^{-1}(y)}^{-1}(t)\big)\big),~\textrm{by}~\eqref{useful-in-subgroup}\\
  &=&    \big(\phi^{-1}(x), ~\tau_{\phi^{-1}(x)}^{-1}(s)\big)* \big(\phi^{-1}(y), ~\tau_{\phi^{-1}(y)}^{-1}(t)\big)\\
  &=&  \varphi(x,s) * \varphi(y,t).
\end{eqnarray*}
Thus, $\varphi \in \Aut^{x_0}_S(E)$ and a direct check shows that it is the inverse of $\psi$. Further, given $\psi(x, s)=\big(\phi(x), ~\tau_x(s)\big)$ and $\psi'(x, s)=\big(\phi'(x),~ \tau'_x(s)\big)$, we see that
\begin{equation}\label{product-psi}
\psi \psi'(x, s)=\psi \big( \phi'(x), ~\tau'_x(s)\big)=\big(\phi\big( \phi'(x)\big), ~\tau_{\phi'(x)}\big(\tau'_x(s)\big) \big)=\big(\phi \phi'(x), ~\tau_{\phi'(x)}\tau'_x(s) \big),
\end{equation}
and hence $\Aut^{x_0}_S(E)$ is a subgroup of $\Aut(E)$.
\end{proof}
\medskip

Restricting the action of $\Aut(X) \times \Sigma_S$ to that of its subgroup $\Aut^{x_0}(X) \times \Sigma_S$, let $$\Omega_{[\alpha]}: \Aut^{x_0}(X) \times \Sigma_S \to \mathcal{H}^2(X; S)$$ be the orbit map given by 
\begin{equation}\label{wells-map}
\Omega_{[\alpha]}(\phi, \theta)= ~^{(\phi, \theta)}[\alpha].
\end{equation}
For simplicity of notation, we denote  $\Omega_{[\alpha]}$ by $\Omega$. Further, let $$\Phi: \Aut^{x_0}_S(E) \to \Aut^{x_0}(X) \times \Sigma_S$$ be the restriction map given by 
\begin{equation}\label{restriction-map}
\Phi(\psi)= ~(\phi, \theta),
\end{equation} 
where $\psi(x, s)=\big(\phi(x),~ \tau_x(s)\big)$ and $\theta(s):=\tau_{x_0}(s)$ for all $x \in X$ and $s \in S$.  Define a map $\lambda:X \to \Sigma_S$ by setting $$\lambda_x:= \theta\tau_{\phi^{-1}(x)}^{-1}.$$ 
Note that $\lambda_{x_0}=\id_S$ and $\tau_x$ can be uniquely written as 
 \begin{equation}\label{expression-tau}
 \tau_x=(\tau_x \theta^{-1}) \theta=\lambda_{\phi(x)}^{-1}\theta
\end{equation}
for all $x \in X$.

\begin{proposition}\label{image-rho}
The map $\Phi$ is a group homomorphism and $\img(\Phi)=\Omega^{-1}\big([\alpha]\big)$.
\end{proposition}

\begin{proof}
Let $\psi, \psi' \in \Aut^{x_0}_S(E)$. Then, by \eqref{product-psi}, we have $$\psi \psi'(x, s)=\big(\phi \phi'(x), ~\tau_{\phi'(x)}\tau'_x(s) \big),$$
where $\tau_{\phi'(x_0)}\tau'_{x_0}(s)=\tau_{x_0}\tau'_{x_0}(s)=\theta \theta'(s)$. This shows that $\Phi(\psi \psi' )=\Phi(\psi)\Phi(\psi' )$.
\par
Notice that $$\Omega^{-1}\big([\alpha]\big)= \big( \Aut^{x_0}(X) \times \Sigma_S \big)_{[\alpha]},$$ the stabiliser of the action of $\Aut^{x_0}(X) \times \Sigma_S$ at $[\alpha]$. Let $\psi \in \Aut^{x_0}_S(E)$ such that $\Phi (\psi)= (\phi, \theta)$. Since $\psi$ is a quandle homomorphism, for $x, y \in X$ and $s, t \in S$, we have
\begin{equation}\label{ext-homo-condition1}
\tau_{x*y}\big(\alpha_{x, y}(s, t)\big)=\alpha_{\phi(x), \phi(y)}\big(\tau_x(s),~ \tau_y(t)\big).
\end{equation}
Since $\phi \in \Aut^{x_0}(X)$ and $\theta \in \Sigma_S$, we can replace $x, y$ by $\phi^{-1}(x),\phi^{-1}(y)$ and $s, t$ by $\theta^{-1}(s),\theta^{-1}(t)$, respectively,  in \eqref{ext-homo-condition1}. This gives
\begin{equation}\label{ext-homo-condition2}
\tau_{\phi^{-1}(x*y)}\Big(\alpha_{\phi^{-1}(x), \phi^{-1}(y)}\big(\theta^{-1}(s), ~\theta^{-1}(t)\big)\Big)=\alpha_{x, y}\Big(\tau_{\phi^{-1}(x)}\big(\theta^{-1}(s)\big), ~\tau_{\phi^{-1}(y)}\big(\theta^{-1}(t)\big)\Big).
\end{equation}
Define a map $\lambda: X \to \Sigma_S$ by setting $\lambda_x:=\theta \tau^{-1}_{\phi^{-1}(x)}$. Then \eqref{ext-homo-condition2} takes the form
\begin{equation}\label{ext-homo-condition3}
\lambda_{x*y}^{-1} ~\theta \Big(\alpha_{\phi^{-1}(x), ~\phi^{-1}(y)}\big(\theta^{-1}(s), ~\theta^{-1}(t)\big)\Big)=\alpha_{x, y}\big(\lambda_x^{-1}(s), ~\lambda_y^{-1}(t) \big),
\end{equation}
which is further equivalent to
\begin{equation}\label{ext-homo-condition4}
^{(\phi, \theta)}\alpha_{x, y}(s, t)=\theta \Big(\alpha_{\phi^{-1}(x), ~\phi^{-1}(y)}\big(\theta^{-1}(s), ~\theta^{-1}(t)\big)\Big)= \lambda_{x*y} \Big(\alpha_{x, y}\big(\lambda_x^{-1}(s), ~\lambda_y^{-1}(t) \big) \Big).
\end{equation}
This shows that $\img(\Phi)\subseteq \big( \Aut^{x_0}(X) \times \Sigma_S \big)_{[\alpha]}$.
\par
Conversely, let $(\phi, \theta) \in \big( \Aut^{x_0}(X) \times \Sigma_S \big)_{[\alpha]}$, that is, there exists a map $\lambda: X \to \Sigma_S$ such that \eqref{ext-homo-condition4} holds. Define $\psi: E \to E$ by setting
$$\psi(x, s)= \big(\phi(x),~ \lambda_{\phi(x)}^{-1} \theta (s)\big).$$
Clearly, $\psi$ is bijective since $\phi$ and $\theta$ are so. Replacing $x, y$ by $\phi(x),\phi(y)$ and $s, t$ by $\theta(s),\theta(t)$, respectively, in \eqref{ext-homo-condition4} gives 
\begin{equation}\label{ext-homo-condition5}
\lambda_{\phi(x)*\phi(y)}^{-1}~\theta \big(\alpha_{x, y}(s, ~t)\big)=  \alpha_{\phi(x), \phi(y)}\big(\lambda_{\phi(x)}^{-1}\theta(s), ~\lambda_{\phi(y)}^{-1}\theta(t) \big).
\end{equation}
Now, we compute
\begin{eqnarray*}
\psi \big((x, s)*(y, t) \big) & = & \psi \big(x*y, ~\alpha_{x, y}(s, t) \big) \\
&  =&  \Big(\phi(x)*\phi(y), ~\lambda_{\phi(x)*\phi(y)}^{-1} \theta \big(\alpha_{x, y}(s, t)\big) \Big) \\
&  =&  \Big(\phi(x)*\phi(y), ~\alpha_{\phi(x), \phi(y)}\big(\lambda_{\phi(x)}^{-1}\theta(s), ~\lambda_{\phi(y)}^{-1}\theta(t) \big) \Big),~ \textrm{by}~\eqref{ext-homo-condition5}\\
&  =&  \big(\phi(x), ~\lambda_{\phi(x)}^{-1}\theta(s) \big)* \big(\phi(y), ~\lambda_{\phi(y)}^{-1}\theta(t) \big)\\ 
&  =& \psi (x, s)*\psi (y, t).
\end{eqnarray*}
Thus, $\psi \in\Aut^{x_0}_S(E)$ and $\Phi(\psi)=(\phi, \theta)$, which completes the proof.
\end{proof}

\begin{proposition}\label{kernel-rho}
$\Ker(\Phi) \cong \Big\{\lambda: X \to \Sigma_S~|~\lambda_{x_0}=\id_S~\textrm{and}~\alpha_{x, y}(s, t)=\lambda_{x*y} \big(\alpha_{x, y}\big(\lambda_x^{-1}(s),\lambda_y^{-1}(t)\big)\big) \Big\}$.
\end{proposition}
\begin{proof}
Note that $\psi \in \Ker(\Phi)$ if and only if $\Phi(\psi)=(\id_X, \id_S)$. In view of \eqref{expression-tau}, we have $$\psi(x, s)=\big(x, \lambda_x^{-1}(s)\big),$$ and $\lambda_{x_0}=\id_S$. Further, by \eqref{ext-homo-condition1}, $\psi$ is a quandle homomorphism if and only if $$\alpha_{x, y}(s, t)=\lambda_{x*y} \big(\alpha_{x, y}\big(\lambda_x^{-1}(s),\lambda_y^{-1}(t)\big)\big)$$ for all $x, y \in X$ and $s,t \in S$. Thus, $\Ker(\Phi)$ is the desired set via the map $\psi \mapsto \lambda$. Here, in view of \eqref{product-psi}, the right hand side is a group with operation $(\lambda\lambda')_x :=\lambda_x \lambda'_x$ for all $x \in X$.
\end{proof}

Combining propositions \ref{image-rho} and \ref{kernel-rho}, we obtain our first main result.

\medskip

\begin{theorem}\label{main-thm-1}
Let $X$ be a quandle, $x_0 \in X$, $S$ a set, $\alpha\in \mathcal{Z}^2(X; S)$ and $E=X \times_\alpha S$ the extension of $X$ by $S$ using $\alpha$. Then there exists an exact sequence
\begin{equation}\label{well-sequence}
1 \longrightarrow \Ker(\Phi) \longrightarrow \Aut^{x_0}_S(E) \stackrel{\Phi}{\longrightarrow} \Aut^{x_0}(X) \times \Sigma_S \stackrel{\Omega}{\longrightarrow} \mathcal{H}^2(X; S),
\end{equation}
where exactness at $ \Aut^{x_0}(X) \times \Sigma_S$ means that $\img(\Phi)=\Omega^{-1}\big([\alpha]\big)$.
\end{theorem}
\medskip

The preceding theorem can be thought of as a quandle analogue of a similar result of Wells for groups \cite{MR0272898}. As a consequence, we obtain the following short exact sequence.

\begin{corollary}
Let $X$ be a quandle, $x_0 \in X$, $S$ a set, $\alpha\in \mathcal{Z}^2(X; S)$ and $E=X \times_\alpha S$ the extension of $X$ by $S$ using $\alpha$. Then there exists a short exact sequence of groups
\begin{equation}\label{der-aut-short-exact}
1 \longrightarrow \Ker(\Phi) \longrightarrow \Aut^{x_0}_S(E) \stackrel{\Phi}{\longrightarrow} \big(\Aut^{x_0}(X) \times \Sigma_S\big)_{[\alpha]} \longrightarrow 1.
\end{equation}
\end{corollary}
\medskip

It is interesting to find conditions for the splitting of the short exact sequence  \eqref{der-aut-short-exact}.

\begin{proposition}
Let $X$ and $S$ be quandles and $x_0 \in X$. If $\alpha \in \mathcal{Z}^2(X; S)$ is the cocycle given by $\alpha_{x, y}(s, t)= s*t$ for all $x, y \in X$ and $s, t \in S$, then the short exact  sequence 
\begin{equation}\label{der-aut-short-exact2}
1 \longrightarrow \Ker(\Phi) \longrightarrow \Aut^{x_0}_S(E) \stackrel{\Phi}{\longrightarrow} \big(\Aut^{x_0}(X) \times \Aut(S)\big)_{[\alpha]} \longrightarrow 1
\end{equation}
 splits.
\end{proposition}

\begin{proof}
Note that for the given $\alpha$, the quandle operation \eqref{genralised-quandle-operation} on $E=X \times_\alpha S$ becomes $$(x, s)* (y,t)= (x* y, s* t ),$$ which is just the product quandle structure. Also, a direct check shows that, in this case, 
$$\Ker(\Phi) \cong \Big\{\lambda: X \to \Aut(S)~|~\lambda_{x_0}=\id_S~\textrm{and}~\lambda_{x*y}(s*t)=\lambda_x(s) * \lambda_y(t)\Big\}$$
and
$$\Aut^{x_0}_S(E)=\Big\{ \psi \in \Aut(E)~|~\psi(x, s)=\big(\phi(x),~ \tau_x(s)\big)~\textrm{for some}~\phi \in \Aut^{x_0}(X)~\textrm{and}~\tau: X \to \Aut(S) \Big\}.$$
Now, for any $(\phi, \theta) \in \Aut^{x_0}(X) \times \Aut(S)$, we have
$$^{(\phi, \theta)}\alpha_{x, y}(s, t)= \theta \Big(\alpha_{\phi^{-1}(x), \phi^{-1}(y)}\big(\theta^{-1}(s), ~\theta^{-1}(t)\big)  \Big)=\theta \big(\theta^{-1}(s)*\theta^{-1}(t)\big)= s*t=\alpha_{x, y}(s, t)$$
for $x, y \in X$ and $s, t \in S$. Thus, $\big(\Aut^{x_0}(X) \times \Aut(S) \big)_{[\alpha]}=\Aut^{x_0}(X) \times \Aut(S)$. Define 
$$\zeta:  \Aut^{x_0}(X) \times \Aut(S) \to \Aut^{x_0}_S(E)$$ by setting
$\zeta(\phi, \theta)=\psi$, where $\psi(x, s)= \big(\phi(x), \theta(s)\big)$ for $x \in X$ and $s \in S$. Then $\zeta$ is a group homomorphism with $\Phi \circ \zeta$ being the identity map, and sequence \eqref{der-aut-short-exact2} splits.
\end{proof}

\begin{remark}
We believe that an analogue of Theorem \ref{main-thm-1} should be provable for the general cohomology theory for quandles developed in \cite{Andruskiewitsch2003}. 
\end{remark}
\medskip

\section{Exact sequence relating automorphisms and second quandle cohomology}\label{abelian-quandle-cohomology-section}
Specialising to the usual quandle cohomology with coefficients in an abelian group, after some appropriate modifications, the sequence \eqref{well-sequence} takes a simpler and elegant form, which we derive in this section. We briefly recall the cohomology theory for quandles as in \cite{CJKLS2003}. Let $X$ be a quandle and $A$ an abelian group. Let $S_n(X)$ be the free abelian group generated by $n$-tuples of elements of $X$. Define a homomorphism
$\partial_{n}: S_{n}(X) \to S_{n-1}(X)$ by 

\begin{eqnarray*}
\partial_{n}(x_1, x_2, \dots, x_n)  & = & \sum_{i=2}^{n} (-1)^{i} \big( (x_1, x_2, \dots, x_{i-1}, x_{i+1},\dots,  x_n)\\
& & - (x_1 \ast x_i, x_2 \ast x_i, \dots, x_{i-1}\ast x_i, x_{i+1},x_{i+2},  \dots, x_n) \big)
\end{eqnarray*}
for $n \geq 2$ and $\partial_n=0$ for $n \leq 1$. Then $\{S_n(X), \partial_n \}$ forms a chain complex. For $n \geq 2$, let $D_n(X)$ be the free abelian subgroup of $S_n(X)$ generated by $n$-tuples $(x_1, \dots, x_n)$ with $x_{i}=x_{i+1}$ for some $1 \le i \le n-1$. Set $D_n(X)=0$   for $n \leq 1$. It can be checked that if $X$ is a quandle, then $\partial_n\big(D_n(X)\big) \subseteq D_{n-1}(X)$. Setting $C_n(X) = S_n(X)/ D_n(X)$, we see that $\{ C_n(X), \partial_n \}$ forms a chain complex,  where by abuse of notation, $\partial_n$ also denotes the induced homomorphism. 
\medskip

Define $C^n(X;A) := \Hom \big(C_n(X), A\big)$ and $\delta^n :C^n(X;A) \to C^{n+1}(X;A)$ given by
$$\delta^n(f)= (-1)^nf \circ \partial_{n+1}.$$ 
This turns $\{C^n(X;A), \delta^n \}$ into a cochain complex, and its $n$-th cohomology group $\Ho^n(X;A)$ is called the $n$-th quandle cohomology group of  $X$ with coefficients in $A$. We denote the group of cocycles and the group of coboundaries by  $\Z^n(X;A)$ and $\B^n(X;A)$, respectively.
\medskip

As in Remark \ref{abelian-remark}, $\Z^2(X;A)$ can be identified with the set of maps $\alpha: X \times X \to A$ satisfying the 2-cocycle condition \eqref{group-coefficient-cocycle-condition} and \eqref{normalised-cocycle-condition}, that is, 
$$\alpha_{x, y}~\alpha_{x*y, z}= \alpha_{x, z}~\alpha_{x* z, y*z}$$
and $$\alpha_{x, x}=1$$
for $x, y, z \in X$.  Further, two 2-cocycles $\alpha, \beta \in \Z^2(X;A)$ are cohomologous if there exists a 1-cochain $\lambda \in C^1(X;A)$ such that $\beta ~\delta^1 (\lambda)= \alpha $, that is, $$\beta_{x, y}=\lambda_{x*y} \lambda_x^{-1}\alpha_{x, y}$$ for $x, y \in X$.
\medskip

For $(\phi, \theta) \in \Aut(X) \times \Aut(A)$  and $\alpha \in \Z^2(X;A)$, we define
\begin{equation}\label{abelian-action}
^{(\phi, \theta)}\alpha_{x, y}:= \theta \big(\alpha_{\phi^{-1}(x), \phi^{-1}(y)}  \big)
\end{equation}
for $x,y \in X$. 

\begin{proposition}\label{aut-action-by-auto-coho}
The group $\Aut(X) \times \Aut(A)$ acts by automorphisms on $\Ho^2(X; A)$ as
$$^{(\phi, \theta)}[\alpha] := [^{(\phi, \theta)}\alpha]$$ 
for $(\phi, \theta) \in \Aut(X) \times \Aut(A)$ and $[\alpha]\in \Ho^2(X; A)$.
\end{proposition}

\begin{proof}
That  \eqref{abelian-action} defines an action of $\Aut(X) \times \Aut(A)$ on $\Z^2(X;A)$ is easy and follows along the lines of proof of Lemma \ref{aut-action-cocycles}. Let $(\phi, \theta) \in \Aut(X) \times \Aut(A)$  and $\alpha, \beta \in \Z^2(X;A)$. Then, for $x, y \in X$, we have
\begin{eqnarray*}
^{(\phi, \theta)}(\alpha \beta)_{x, y} &= & \theta \big({(\alpha\beta)}_{\phi^{-1}(x), \phi^{-1}(y)}  \big)\\
&= & \theta \big(\alpha_{\phi^{-1}(x), \phi^{-1}(y)}~\beta_{\phi^{-1}(x), \phi^{-1}(y)}  \big)\\
&= & \theta \big(\alpha_{\phi^{-1}(x), \phi^{-1}(y)}\big) ~ \theta\big(\beta_{\phi^{-1}(x), \phi^{-1}(y)}  \big), ~\textrm{since}~\theta \in \Aut(A)\\
&= &  ^{(\phi, \theta)}\alpha_{x, y} ~^{(\phi, \theta)}\beta_{x, y}\\
&= &  (^{(\phi, \theta)}\alpha ~^{(\phi, \theta)}\beta)_{x, y},
\end{eqnarray*}
and hence the action is by automorphisms. Let $\alpha,\beta \in \Z^2(X;A)$ be two cohomologous quandle 2-cocycles. Then there exists a map $\lambda:X \to A$ such that $$\beta_{x, y}=\alpha_{x, y}~\lambda_x~ \lambda_{x*y}^{-1}.$$
Setting $\lambda'_x:=\theta (\lambda_{\phi^{-1}(x)})$, we see that
\begin{eqnarray*}
^{(\phi, \theta)}\beta_{x, y} &=& \theta \big(\beta_{\phi^{-1}(x), \phi^{-1}(y)}  \big)\\
 &=& \theta \big(\alpha_{\phi^{-1}(x), \phi^{-1}(y)} \lambda_{\phi^{-1}(x)} \lambda_{\phi^{-1}(x*y)}^{-1} \big)\\
 &=& ^{(\phi, \theta)}\alpha _{x, y} ~\theta \big(\lambda_{\phi^{-1}(x)} \big) \theta \big(\lambda_{\phi^{-1}(x*y)}^{-1} \big)\\
 &=& ^{(\phi, \theta)}\alpha _{x, y} ~ \lambda'_x ~{\lambda'_{x*y}}^{-1}.
\end{eqnarray*}
Thus, $\Aut(X) \times \Aut(A)$ acts by automorphisms on  $\Ho^2(X; A)$.
\end{proof}

Applying the orbit-stabiliser theorem to the action of $\Aut(X) \times \Aut(A)$ on $\Ho^2(X; A)$ yields the following bound on the size of second quandle cohomology.

\begin{corollary}
If $X$ is a finite quandle, $A$ a finite abelian group and $\alpha: X\times X \to A$ a quandle 2-cocycle, then
$$|\Ho^2(X; A)| \ge \frac{|\Aut(X) \times \Aut(A)|}{|\big(\Aut(X) \times \Aut(A)\big)_{[\alpha]}|},$$
where $\big(\Aut(X) \times \Aut(A)\big)_{[\alpha]}$ is the stabiliser of $\Aut(X) \times \Aut(A)$ at $[\alpha]$.
\end{corollary}
\medskip

Let $[\alpha] \in\Ho^2(X; A)$ be a fixed cohomology class. Note that the group $\Ho^2(X; A)$ acts freely and transitively on itself by left multiplication. For any $(\phi, \theta) \in \Aut(X) \times \Aut(A)$, since $[\alpha]$ and $^{(\phi, \theta)}[\alpha]$ are two elements of $\Ho^2(X; A)$, there exists a unique element, say, $\Theta_{[\alpha]} (\phi, \theta) \in \Ho^2(X; A)$ such that
\begin{equation}\label{abelian-wells-map}
^{\Theta_{[\alpha]}  (\phi, \theta)} \big(^{(\phi, \theta)}[\alpha] \big)= [\alpha].
\end{equation}
This defines a map $$\Theta_{[\alpha]} : \Aut(X) \times \Aut(A) \to \Ho^2(X; A).$$ 
For convenience of notation, we denote $\Theta_{[\alpha]}$ by $\Theta$. By Remark \ref{abelian-remark}, the map $\alpha': X \times X \to \Map(A\times A, A)$ given by $\alpha'_{x, y}(s, t)=s~\alpha_{x, y}$, where $x, y \in X$ and $s, t \in A$,  is a dynamical 2-cocycle. Thus, by Proposition \ref{set-cocycle}, the binary operation
\begin{equation}\label{abelian-extension-operation}
(x, s)*(y, t):=(x*y,~s~\alpha_{x, y})
\end{equation}
defines a quandle $E:= X \times_\alpha A$ called the {\it abelian extension} of $X$ by $A$ using $\alpha$. We would like to understand certain group of automorphisms of the quandle $E$ in terms of $\Aut(X)$, $\Aut(A)$ and  $\Ho^2(X; A)$. For this purpose, we define
\begin{small}
$$\Aut_A(E)=\Big\{ \psi \in \Aut(E)~|~\psi(x, s)=\big(\phi(x),~ \lambda_x\theta(s)\big)~\textrm{for some}~(\phi, \theta) \in \Aut(X)\times \Aut(A)~\textrm{and map}~\lambda:X \to A \Big\}.$$
\end{small}

\begin{proposition}
$\Aut_A(E)$ is a subgroup of $\Aut(E)$.
\end{proposition}

\begin{proof}
Let $\psi \in\Aut_A(E)$ such that $\psi(x, s)=\big(\phi(x),~ \lambda_x\theta(s)\big)$ for $x \in X$ and $s \in A$. Define $\varphi(x,s) := \big(\phi^{-1}(x), ~\theta^{-1} \big(\lambda_{\phi^{-1}(x)}^{-1}s\big)\big)$. Clearly, $\varphi$ is a bijection since $\phi$ and $\theta$ are so. Since $\psi$ is a quandle homomorphism, for $x, y \in X$ and $s, t \in S$, we have
\begin{equation}\label{abelian-homo-condition}
\lambda_{x*y} ~\theta (\alpha_{x, y})=\lambda_x~ \alpha_{\phi(x), \phi(y)},
\end{equation}
which is equivalent to
\begin{equation}\label{useful-in-abelian-subgroup}
\theta^{-1} \big(\lambda_{\phi^{-1}(x*y)}^{-1} ~\alpha_{x, y}\big)=\theta^{-1} \big(\lambda_{\phi^{-1}(x)}^{-1}\big)  ~\alpha_{\phi^{-1}(x), \phi^{-1}(y)}.
\end{equation}

Now, consider
\begin{eqnarray*}
\varphi \big( (x, s)*(y, t) \big) &=& \varphi \big(x*y,~ s ~\alpha_{x, y} \big)\\
 &=&  \big(\phi^{-1}(x*y), ~\theta^{-1} \big(\lambda_{\phi^{-1}(x*y)}^{-1} s ~\alpha_{x, y}\big)\big)\\
  &=&    \big(\phi^{-1}(x)*\phi^{-1}(y), ~\theta^{-1} \big(\lambda_{\phi^{-1}(x)}^{-1}s\big)  ~\alpha_{\phi^{-1}(x), \phi^{-1}(y)}  \big),\\
  & &~\textrm{due to}~ \eqref{useful-in-abelian-subgroup}~\textrm{and the fact that $A$ is abelian}\\
  &=&    \big(\phi^{-1}(x), ~\theta^{-1} \big(\lambda_{\phi^{-1}(x)}^{-1}s\big)\big)* \big(\phi^{-1}(y), ~\theta^{-1} \big(\lambda_{\phi^{-1}(y)}^{-1}t\big) \big)\\
  &=&  \varphi(x,s) * \varphi(y,t).
\end{eqnarray*}
Thus, $\varphi \in \Aut_A(E)$ and a direct check shows that $\varphi$ is the inverse of $\psi$. Further, given $\psi(x, s)=\big(\phi(x), ~\lambda_x \theta(s)\big)$ and $\psi'(x, s)=\big(\phi'(x),~ \lambda'_x \theta'(s)\big)$, we have
\begin{equation}\label{product-psi-abelian}
\psi \psi'(x, s)=\psi \big( \phi'(x), ~\lambda'_x \theta'(s) \big)=\big(\phi\big( \phi'(x)\big), ~\lambda_{\phi'(x)} \theta \big(\lambda'_x \theta'(s) \big) \big)=\big(\phi \phi'(x), ~\lambda_{\phi'(x)} ~\theta (\lambda'_x) ~\theta\theta'(s) \big),
\end{equation}
and hence $\Aut_A(E)$ is a subgroup of $\Aut(E)$.
\end{proof}
\medskip

In this case, we have a natural restriction map $$\Psi: \Aut_A(E) \to \Aut(X) \times \Aut(A)$$ given by $\Psi(\psi)=(\phi, \theta)$.

\begin{proposition}\label{image-rho-ker-Theta}
The map $\Psi$ is a group homomorphism and $\img(\Psi)=\Theta^{-1}(1)$.
\end{proposition}

\begin{proof}
It is clear from \eqref{product-psi-abelian} that $\Psi$ is a group homomorphism. First notice that $$\Theta^{-1}(1)=\big(\Aut(X) \times \Aut(A) \big)_{[\alpha]}$$ and let $(\phi, \theta)\in \big(\Aut(X) \times \Aut(A) \big)_{[\alpha]}$. Then there exists a map $\lambda:X \to A$ such that
\begin{equation}\label{showing-psi-kernel}
\theta \big(\alpha_{\phi^{-1}(x), \phi^{-1}(y)} \big)=\alpha_{x, y}~\lambda_x~ \lambda_{x*y}^{-1}.
\end{equation}
for all $x, y \in X$. Define $\psi: E \to E$ by setting $\psi(x, s)= \big(\phi(x), \lambda_{\phi(x)} \theta(s)\big)$. Clearly, $\psi$ is bijective and is a quandle homomorphism due to \eqref{showing-psi-kernel}. Further, $\Psi(\psi)=(\phi, \theta)$, and hence $\big(\Aut(X) \times \Aut(A) \big)_{[\alpha]} \subseteq \img(\Psi)$.
\par
Conversely, if $(\phi, \theta)\in \img(\Psi)$, then there exists a map $\lambda: X \to A$ such that \eqref{showing-psi-kernel} holds. But this is equivalent to $^{(\phi, \theta)}[\alpha]=[\alpha]$, which is desired.
 \end{proof}
\medskip

\begin{proposition}\label{kernel-rho-abelian}
$\Ker(\Psi) \cong \Z^1(X;A) =\{\lambda: X \to A~|~\lambda_x=\lambda_{x*y}~\textrm{for all}~x, y \in X\}$.
\end{proposition}

\begin{proof}
Note that $\psi \in \Ker(\Psi)$ if and only if $\Psi(\psi)=(\id_X, \id_S)$. This gives $\psi(x, s)=\big(x, \lambda_x~ s\big)$ for all $x \in X$ and $s \in A$. Further, by \eqref{abelian-homo-condition}, $\psi$ is a quandle homomorphism if and only if $\lambda_x=\lambda_{x*y}$. Thus, $\Ker(\Psi)$ is the desired group $\Z^1(X;A)$ of 1-cocycles identified via the map $\psi \mapsto \lambda$.
\end{proof}

\medskip
We are now in a position to state the main result of this section that gives an exact sequence relating quandle 1-cocycles, automorphisms and second cohomology of quandles for an abelian quandle extension.

\begin{theorem}\label{abelian-main-theorem}
Let $X$ be a quandle, $A$ an abelian group, $\alpha\in \Z^2(X; A)$ and $E=X \times_\alpha A$ the abelian extension of $X$ by $A$ using $\alpha$. Then there exists an exact sequence
\begin{equation}\label{abelian-well-sequence}
1 \longrightarrow \Z^1(X;A) \longrightarrow \Aut_A(E) \stackrel{\Psi}{\longrightarrow} \Aut(X) \times \Aut(A) \stackrel{\Theta}{\longrightarrow} \Ho^2(X; A),
\end{equation}
where exactness at $ \Aut(X) \times \Aut(A)$ means that $\img(\Psi)=\Theta^{-1}(1)$.
\end{theorem}

Theorem \ref{abelian-main-theorem} can be reformulated as follows.

\begin{corollary}
Let $X$ be a quandle, $A$ an abelian group, $\alpha\in \Z^2(X; A)$ and $E=X \times_\alpha A$ the abelian extension of $X$ by $A$ using $\alpha$. Then there exists a short exact sequence
\begin{equation}\label{abelian-well-sequence}
1 \longrightarrow \Z^1(X;A) \longrightarrow \Aut_A(E) \stackrel{\Psi}{\longrightarrow} \big(\Aut(X) \times \Aut(A)\big)_{[\alpha]} \longrightarrow 1.
\end{equation}
\end{corollary}

The second quandle cohomology is an obstruction to extension of automorphisms in quandle extensions.

\begin{corollary}
Let $X$ be a quandle and $A$ an abelian group such that $\Ho^2(X; A)=1$. Then every automorphism in $\Aut(X) \times \Aut(A)$ extends to an automorphism in $ \Aut_A(E)$.
\end{corollary}

Restricting the action of $\Aut(X) \times \Aut(A)$ on $\Ho^2(X; A)$ to that of its subgroups $\Aut(X)$ and $\Aut(A)$ gives the following result.

\begin{corollary}
Every automorphism in $\Aut(X)_{[\alpha]}$ and $\Aut(A)_{[\alpha]}$ can be extended to an automorphism in $\Aut_A(E)$.
\end{corollary}
\medskip

\section{Properties of the map $\Theta$}\label{properties of map theta}
In this section, we study the map $\Theta$ in more detail. Let $X$ be a quandle and $A$ an abelian group. Since the group $\Aut(X) \times \Aut(A)$ acts on the group $\Ho^2(X; A)$, we have their semi-direct product $\Ho^2(X; A) \rtimes \big(\Aut(X) \times \Aut(A) \big)$. Further, note that the group $\Ho^2(X; A)$ acts on itself by left multiplication.

\begin{proposition}
For $(\phi, \theta) \in  \Aut(X) \times \Aut(A)$ and $[\alpha], [\beta]\in  \Ho^2(X; A)$ setting
$$^{[\alpha](\phi, \theta)} [\beta]=~^{[\alpha]}{(^{(\phi, \theta)} [\beta])}$$
defines an action of $~\Ho^2(X; A) \rtimes \big(\Aut(X) \times \Aut(A) \big)$ on $\Ho^2(X; A)$.
\end{proposition}

\begin{proof}
For $ (\phi_1, \theta_1), (\phi_2, \theta_2) \in \Aut(X) \times \Aut(A)$ and $[\alpha_1], [\alpha_2], [\beta] \in \Ho^2(X; A)$, we compute
\begin{eqnarray*}
^{\big([\alpha_1](\phi_1, \theta_1)\big)\big([\alpha_2](\phi_2, \theta_2)\big)}[\beta] & = & ^{\big([\alpha_1]~^{(\phi_1, \theta_1)}[\alpha_2] \big) \big((\phi_1, \theta_1)(\phi_2, \theta_2)\big)}[\beta]\\
& = & ^{\big([\alpha_1]~^{(\phi_1, \theta_1)}[\alpha_2]\big)}{\big(^{\big((\phi_1, \theta_1)(\phi_2, \theta_2)\big)}[\beta] \big)}\\
& = & ^{[\alpha_1]}{\big(^{^{(\phi_1, \theta_1)}{[\alpha_2]}}{\big(^{(\phi_1, \theta_1)}{\big(^{(\phi_2, \theta_2)}[\beta]\big)} \big)} \big)}\\
& = & ^{[\alpha_1]}{\big({^{(\phi_1, \theta_1)}{[\alpha_2]}}{\big(^{(\phi_1, \theta_1)}{\big(^{(\phi_2, \theta_2)}[\beta]\big)} \big)} \big)},\\
& &  \textrm{since}~^{(\phi_1, \theta_1}{[\alpha_2]} \in  \Ho^2(X; A),~\textrm{which acts on itself by left translation}\\
& = & ^{[\alpha_1]}{\big({^{(\phi_1, \theta_1)}{\big([\alpha_2] \big(^{(\phi_2, \theta_2)}[\beta]\big)}} \big)\big)},\\
&  & \textrm{since}~\Aut(X) \times \Aut(A)~\textrm{acts by automorphisms on}~\Ho^2(X; A)\\
& = & ^{\big([\alpha_1](\phi_1, \theta_1)\big)}{\big(^{\big([\alpha_2](\phi_2, \theta_2)\big)}[\beta] \big)},\\
& & ~ \textrm{since}~\Ho^2(X; A)~\textrm{acts on itself by left translation}.
\end{eqnarray*}
Thus, $\Ho^2(X; A) \rtimes \big(\Aut(X) \times \Aut(A) \big)$ acts on $\Ho^2(X; A)$.
\end{proof}
\medskip

We recall two basic definitions from group cohomology that we need in the subsequent discussion \cite[Chapter 4]{Brown1981}. Let $A$ be an abelian group and $G$ a group acting on $A$ from left by automorphisms.  Then the first cohomology of $G$ with coefficients in $A$ is defined as $\Ho^1(G; A):=\Z^1(G; A)/\B^1(G; A)$, where
$$\Z^1(G; A) =  \big\{f:G  \to A~|~ f(xy)= f(x)~{^{x}}f(y)~ \textrm{for all}\ x,y\in G\big\}$$
is the group of {\it derivations} or  {\it 1-cocycles} and
$$\B^1(G;A) = \big\{f:G\to A~|~\textrm{there exists}~a \in A~\textrm{such that}~ f(x)=({^x}a)a^{-1}~\textrm{for all}~x\in G \big\}$$
is the group of {\it inner derivations} or  {\it 1-coboundaries}.
\medskip

Also recall that a complement of a subgroup $H$ in a group $G$ is another subgroup $K$ of $G$ such that $G=HK$ and $H\cap K=1$.  The following result relating derivations and complements is well-known \cite[11.1.2]{Robinson1982}.

\begin{lemma}\label{1-cocycle-complement}
Let $H$ be an abelian group and $G$ a group acting on $H$ by automorphisms. Then the map $f \mapsto \{f(g)g~|~g \in G \}$ gives a bijection from the set $\Z^1(G; H)$ of derivations to the set $\{K~|~G=HK~\textrm{and}~H \cap K=1 \}$ of complements of $H$ in $G$. 
\end{lemma}

\begin{theorem}\label{main-thm-3}
The map $\Theta: \Aut(X) \times \Aut(A) \longrightarrow \Ho^2(X; A)$ is a derivation. Further, if $\Theta$ and $\Theta'$ correspond to two  quandle 2-cocycles, then $\Theta$ and $\Theta'$ differ by an inner derivation.
\end{theorem}

\begin{proof}
Suppose that $\Theta=\Theta_{[\alpha]}$ for $[\alpha] \in \Ho^2(X; A)$ and $\bold{g} \in \Ho^2(X; A) \rtimes \big( \Aut(X) \times \Aut(A)\big)$. Then for elements $[\alpha], ~^{\bold{g}}[\alpha] \in \Ho^2(X; A)$, there exists a unique $[\beta] \in \Ho^2(X; A)$ such that $~^{[\beta]}[\alpha]=~^{\bold{g}}[\alpha]$. This shows that $[\beta]^{-1}\bold{g} \in \mathbb{H}$, the stabiliser subgroup of $\Ho^2(X; A) \rtimes \big( \Aut(X) \times \Aut(A) \big)$ at $[\alpha]$, and hence
$$ \Ho^2(X; A) \rtimes \big( \Aut(X) \times \Aut(A) \big)= \Ho^2(X; A)\mathbb{H}.$$ Further, since $\Ho^2(X; A)$ acts freely on itself, it follows that $\mathbb{H}$ is the complement of $\Ho^2(X; A)$ in the group $\Ho^2(X; A) \rtimes \big(\Aut(X) \times \Aut(A) \big)$. By Lemma \ref{1-cocycle-complement}, let $f:\Aut(X) \times \Aut(A) \to \Ho^2(X; A)$ be the unique derivation corresponding to the complement $\mathbb{H}$ of $\Ho^2(X; A)$ in $\Ho^2(X; A) \rtimes \big(\Aut(X) \times \Aut(A)\big)$. Then 
$$\mathbb{H}= \big\{f(\phi, \theta)(\phi, \theta)~|~(\phi, \theta) \in \Aut(X) \times \Aut(A) \big\},$$
that is, $$[\alpha]=~^{f(\phi, \theta)(\phi, \theta)}[\alpha]=~^{f(\phi, \theta)}{\big(^{(\phi, \theta)}[\alpha]\big)}.$$
Now, by definition of $\Theta$ as in \eqref{abelian-wells-map}, we obtain $f(\phi, \theta)=\Theta (\phi, \theta)$, and hence $\Theta$ is a derivation.
\par

Let $\Theta=\Theta_{[\alpha]}$ and $\Theta'=\Theta'_{[\alpha']}$, where $[\alpha], [\alpha'] \in \Ho^2(X; A)$. Then for any $(\phi, \theta) \in \Aut(X) \times \Aut(A)$, we have 
$$ ^{\Theta(\phi, \theta)} \big(^{(\phi, \theta)}[\alpha] \big)= [\alpha]~\textrm{and}~ ^{\Theta'(\phi, \theta)} \big(^{(\phi, \theta)}[\alpha'] \big)= [\alpha'].$$
Since $\Ho^2(X; A)$ acts transitively on itself by left multiplication, there exists a unique $[\beta] \in \Ho^2(X; A)$ such that $ ^{[\beta]}[\alpha']=[\alpha]$. This gives 
$$ ^{[\beta]^{-1}}{\big(^{\Theta(\phi, \theta)} \big(^{^{(\phi, \theta)}{[\beta]}}{\big(^{(\phi, \theta)}[\alpha'] \big)} \big)\big)}= ^{\Theta'(\phi, \theta)} \big(^{(\phi, \theta)}[\alpha'] \big).$$
Since $[\beta]^{-1} \Theta(\phi, \theta) {^{(\phi, \theta)}{[\beta]}}$, $\Theta'(\phi, \theta)$ and  $^{(\phi, \theta)}[\alpha']$ all lie in $\Ho^2(X; A)$, which acts freely on itself, we must have $$[\beta]^{-1} \Theta(\phi, \theta) {^{(\phi, \theta)}{[\beta]}}= \Theta'(\phi, \theta).$$
Thus, $\Theta$ and $\Theta'$ differ by an inner derivation, which completes the proof.
\end{proof}
\medskip

\section{Adjoint functors between categories of quandles and groups}\label{adjoint-functor-section}
Let $\mathcal{Q}$ and $\mathcal{G}$ denote the category of quandles and groups, respectively.  Let $w = w(x,y)$ be a word in the free group $F(x,y)$ on generators $x$ and $y$. For any group $G$, the word $w$ defines a map $G\times G\to G$ given by $(g, h) \mapsto w(g, h)$  for $g, h \in G$. Setting $$g * h = w(g, h)$$ gives an algebraic system $(Q_w(G), *)$. If $w(x,y) = y^{-n} x y^n$ for some integer $n$, then $Q_w(G)$ is the $n$-th conjugation quandle $\Conj_n(G)$, and if $w(x,y) = y x^{-1} y$, then $Q_w(G)$ is the core quandle $\Core(G)$. In fact, the following result \cite[Proposition 3.1]{BarTimSin2019} shows that these are all the possibilities.

\begin{proposition}\label{p7.2}
Let $w=w(x,y)\in F(x,y)$ be such that $Q_w(G)$ is a quandle for every group $G$. Then $w(x,y) = y x^{-1} y$ or $y^{-n} x y^n$ for some $n \in \mathbb{Z}$.
\end{proposition}

For a word $w = w(x,y)\in F(x,y)$ and a quandle $X \in \mathcal{Q}$, we define the group
$$\Adj_w(X)= \big\langle e_x, ~x \in X~\mid~e_{x*y}=w(e_x, e_y),~~x, y \in X \big\rangle.$$
We note that if $w(x,y)=y^{-1}xy$, then $\Adj_w(X)$ is simply denoted by $\Adj(X)$ and referred as the {\it adjoint group} of the quandle $X$ in the literature.

\begin{proposition}\label{verbal-adjoint}
Let $w(x,y) = y x^{-1} y$ or $y^{-n} x y^n$ for some $n \in \mathbb{Z}$. Then $\Adj_w: \mathcal{Q} \to \mathcal{G}$ is a functor that is left adjoint to the functor $Q_w:\mathcal{G} \to \mathcal{Q}$.
\end{proposition}

\begin{proof}
Let $X, Y \in \mathcal{Q}$ and $\phi \in \Hom (X, Y)$ a quandle homomorphism. Then we obtain a group homomorphism
$\Adj_w(\phi): F(X) \to \Adj_w(Y)$ given by $\Adj_w(\phi)(e_x)= e_{\phi(x)}$, where $F(X)$ is the free group on the set $\{e_x~|~x \in X\}$. Further, for $x, y \in X$, we have
\begin{eqnarray*}
\Adj_w(\phi)(e_{x*y})& =& e_{\phi(x)*\phi(y)}\\
& =& w(e_{\phi(x)}, e_{\phi(y)})\\
& =& \Adj_w(\phi) \big(w(e_x, e_y)\big),
\end{eqnarray*}
and hence we obtain a group homomorphism $\Adj_w(\phi): \Adj_w(X) \to \Adj_w(Y).$
\par

Conversely, let $G, H \in \mathcal{G}$ and $f \in \Hom (G, H)$. Defining $Q_w(f): Q_w(G) \to Q_w(H)$ by $Q_w(f)(a)=f(a)$, we see that
\begin{eqnarray*}
Q_w(f)(a*b)& =& f(a*b)\\
& =& f\big(w(a, b)\big)\\
& =& w\big(f(a), f(b)\big)\\
& =& f(a)*f(b)\\
& =& Q_w(f)(a)*Q_w(f)(b)
\end{eqnarray*}
for all $a, b \in G$. Thus, $Q_w(f) \in \Hom \big(Q_w(G), Q_w(H)\big)$. A direct check shows that both $\Adj_w$ and $Q_w$ are functors.
\par

Let $X \in \mathcal{Q}$, $G \in \mathcal{G}$ and $\phi \in \Hom \big(X, Q_w(G)\big)$ a quandle homomorphism. Define $\widetilde{\phi}: F(X) \to G$ by setting $\widetilde{\phi}(e_x)= \phi(x)$. Then, for $x, y \in X$, we have
\begin{eqnarray*}
\widetilde{\phi}(e_{x*y})& =& \phi(x*y)\\
& =& \phi(x)*\phi(y)\\
& =& w\big(\phi(x), \phi(y)\big)\\
& =& \widetilde{\phi}\big(w(e_{x}, e_{y})\big),
\end{eqnarray*}
and hence $\widetilde{\phi} \in \Hom \big(\Adj_w(X), G\big)$. Similarly, let $f \in \Hom \big(\Adj_w(X), G\big)$ be a group homomorphism. Define $\widehat{f}: X \to Q_w(G)$ by setting $\widehat{f}(x)= f(e_x)$. Then, for $x, y \in X$, we have
\begin{eqnarray*}
\widehat{f}(x*y)& =& f(e_{x*y})\\
& =& f\big(w(e_x, e_y)\big)\\
& =& w\big(f(e_x), f(e_y)\big)\\
& =& w\big(\widehat{f}(x), \widehat{f}(y)\big)\\
& =& \widehat{f}(x) * \widehat{f}(y)
\end{eqnarray*}
and hence $\widehat{f} \in \Hom \big(X, Q_w(G)\big)$. It follows easily that $\widehat{\widetilde{\phi}}=\phi$ and $\widetilde{\widehat{f}}=f$ which proves that $\Adj_w$ is left adjoint to $Q_w$.
\end{proof}

Let $\mathcal{Q}'$ denote the category of pairs $(X, \phi)$, where $X \in \mathcal{Q}$ and $\phi \in \Aut(X)$. If $(X_1, \phi_1), (X_2, \phi_2) \in \mathcal{Q}'$, then a morphism from $(X_1, \phi_1)$ to $(X_2, \phi_2)$ is a quandle homomorphism $\psi:X_1 \to X_2$ such that $$\phi_2 \psi= \psi \phi_1.$$ In case of groups, the category $\mathcal{G}'$ is defined analogously. Each object $(G, f) \in \mathcal{G}'$ gives an object $\big(\Alex_f(G),  \Alex(f) \big)$, where $\Alex(f)$ acts as $f$. On the other hand, for $(X, \phi) \in \mathcal{Q}'$, we define the group
$$\Adj_\phi(X)= \big\langle e_x, ~x \in X~\mid~e_{x*y}=e_{\phi(x)}e_{\phi(y)}^{-1}e_{y},~~x, y \in X \big\rangle.$$
Further, the map $\phi$ extends to an automorphism $\Adj(\phi)$ of the free group $F(X)$. For $x, y \in X$, we have 
\begin{eqnarray*}
\Adj(\phi)(e_{x*y})& =& e_{\phi(x)*\phi(y)}\\
& =& e_{\phi^2(x)}e_{\phi^2(y)}^{-1}e_{\phi(y)}\\
& =& \Adj(\phi)(e_{\phi(x)}e_{\phi(y)}^{-1}e_{y}),
\end{eqnarray*}
and hence $\big(\Adj_\phi(X), \Adj(\phi)\big) \in \mathcal{G}'$. It follows easily that both $\Alex$ and $\Adj$ are functors. Imitating the proof of Proposition \ref{verbal-adjoint}, we obtain

\begin{proposition}\label{alexander-adjoint}
The functor $\Adj: \mathcal{Q}' \to \mathcal{G}'$ is left adjoint to the functor $\Alex:\mathcal{G}' \to \mathcal{Q}'$.
\end{proposition}
\medskip

\section{From group extensions to quandle extensions}\label{groups-to-quandles-extensions-section}
We show that extensions of groups give rise to extensions of quandles by applying the functors described in Section \ref{adjoint-functor-section}.

\begin{proposition}\label{core-conj-prop}
Let $1 \to A \to E \to G \to 1$ be an extension of groups and $w(x,y) = y x^{-1} y$ or $y^{-n} x y^n$ for some $n \in \mathbb{Z}$. Then $ Q_w(E) \cong Q_w(G) \times_\alpha Q_w(A)$ for some dynamical 2-cocycle $\alpha$. 
\end{proposition}

\begin{proof}
Let $\pi:E \to G$ be the quotient homomorphism and $\kappa: G \to E$ a transversal such that $\kappa(1)=1$. By functoriality, 
we have a surjective quandle homomorphism $Q_w(\pi): Q_w(E) \to Q_w(G)$ such that
$$\pi \big (\kappa(x)*\kappa(y) \big)= Q_w(\pi) \big (\kappa(x)*\kappa(y) \big)= x*y= Q_w(\pi) \big (\kappa(x*y) \big)= \pi \big (\kappa(x*y) \big)$$
for all $x, y\in G$. Thus, there exists a unique element, say, $\mu(x, y) \in A$ such that $$\kappa(x)*\kappa(y)= \kappa(x*y)~\mu(x, y).$$
Note that $\mu(x, x)=1$ for all $x \in G$. Every element of $E$ can be written uniquely as $\kappa(x)s$ for some $x \in G$ and $s \in A$. For each $x \in G$, there is a bijection $\gamma_x: Q_w(\pi)^{-1}(x) \to A$ given by $\gamma_x\big(\kappa(x)s\big)=s$.  Hence, by \cite[Corollary 2.5]{Andruskiewitsch2003}, we obtain
$$ Q_w(E) \cong Q_w(G) \times_\alpha A,$$ where $\alpha$ is the dynamical 2-cocycle given by $\alpha_{x, y}(s, t)= \gamma_{x*y} \big(\gamma_x^{-1}(s)* \gamma_y^{-1}(t) \big)$ for $x, y \in G$ and $s, t \in A$. 
\par
Case 1. If $w(x,y) = y x^{-1} y$, then
\begin{eqnarray*}
\alpha_{x, y}(s, t) &=& \gamma_{x*y} \big(\gamma_x^{-1}(s)* \gamma_y^{-1}(t) \big)\\
 &=& \gamma_{x*y} \big( \big(\kappa(x)s \big)* \big(\kappa(y)t \big)\big)\\
&=& \gamma_{x*y} \big( \kappa(y)t ~s^{-1}\kappa(x)^{-1} ~\kappa(y)t \big)\\
&=& \gamma_{x*y} \big( \kappa(y)\kappa(x)^{-1} \kappa(y) ~^{\kappa(x)^{-1} \kappa(y)}(t s^{-1}) ~t \big)\\
&=& \gamma_{x*y} \big( \big(\kappa(x)*\kappa(y)\big) ~^{\kappa(x)^{-1} \kappa(y)}(t s^{-1}) ~t \big)\\
&=& \gamma_{x*y} \big( \kappa(x*y)~\mu(x, y) ~^{\kappa(x)^{-1} \kappa(y)}(t s^{-1}) ~t \big)\\
&=&  \mu(x, y) ~^{\kappa(x)^{-1} \kappa(y)}(t s^{-1}) ~t.
\end{eqnarray*}
Recall that the quandle operation in $Q_w(G) \times_\alpha A$ is given by $(x, s)*(y, t)=\big(x*y, \alpha_{x, y}(s, t)\big)$ for $x, y \in G$ and $s, t \in A$. For each fixed $x \in G$, we have $$(x, s)*(x, t)=\big(x, \alpha_{x, x}(s, t)\big)=(x, ts^{-1}t)=(x, s*t),$$ which agrees with the quandle operation in $Q_w(A)$.
\par

Case 2. If $w(x,y) = y^{-n} x y^n$ for some $n \in \mathbb{Z}$, then
\begin{eqnarray*}
\alpha_{x, y}(s, t) &=& \gamma_{x*y} \big( \gamma_x^{-1}(s)* \gamma_y^{-1}(t) \big)\\
 &=& \gamma_{x*y} \big( \big(\kappa(x)s \big)* \big(\kappa(y)t \big)\big)\\
&=& \gamma_{x*y} \big( (\kappa(y)t)^{n} (\kappa(x)s) (\kappa(y)t)^{-n} \big)\\
&=& \gamma_{x*y} \Big( \big(\kappa(y)^n\kappa(x)\kappa(y)^{-n}\big)~\big(^{\kappa(y)^n\kappa(x)^{-1}\kappa(y)^{-n+1}}t\big)~\big(^{\kappa(y)^n\kappa(x)^{-1}\kappa(y)^{-n+2}}t \big)\cdots\\
& & \cdots \big(^{\kappa(y)^n\kappa(x)^{-1}}t \big)~\big(^{\kappa(y)^n}{(st^{-1})} \big)~\big(^{\kappa(y)^{n-1}}{t^{-1}} \big) \cdots \big(^{\kappa(y)^2}{t^{-1}} \big)~\big(^{\kappa(y)}{t^{-1}} \big) \Big)\\
&=& \gamma_{x*y} \Big( \big(\kappa(x)* \kappa(y) \big)~\big(^{\kappa(y)^n\kappa(x)^{-1}\kappa(y)^{-n+1}}t\big)~\big(^{\kappa(y)^n\kappa(x)^{-1}\kappa(y)^{-n+2}}t \big)\cdots\\
& & \cdots \big(^{\kappa(y)^n\kappa(x)^{-1}}t \big)~\big(^{\kappa(y)^n}{(st^{-1})} \big)~\big(^{\kappa(y)^{n-1}}{t^{-1}} \big) \cdots \big(^{\kappa(y)^2}{t^{-1}} \big)~\big(^{\kappa(y)}{t^{-1}} \big) \Big)\\
&=&  \mu(x, y) ~ \Big(\big(^{\kappa(y)^n\kappa(x)^{-1}\kappa(y)^{-n+1}}t\big)~\big(^{\kappa(y)^n\kappa(x)^{-1}\kappa(y)^{-n+2}}t \big)\cdots\\
& & \cdots \big(^{\kappa(y)^n\kappa(x)^{-1}}t \big)~\big(^{\kappa(y)^n}{(st^{-1})} \big)~\big(^{\kappa(y)^{n-1}}{t^{-1}} \big) \cdots \big(^{\kappa(y)^2}{t^{-1}} \big)~\big(^{\kappa(y)}{t^{-1}} \big) \Big).
\end{eqnarray*}
Taking $x=y=1$ and using the facts that $\mu(1, 1)=1=\kappa(1)$, we obtain $$(1, s)*(1, t)=\big(1, \alpha_{1, 1}(s, t)\big)=\big(1, t^{n}st^{-n}\big)=(1, s*t),$$ which agrees with the quandle operation in $Q_w(A)$. Thus, $Q_w(E) \cong Q_w(G) \times_\alpha Q_w(A)$.
\end{proof}

An immediate consequence of Proposition \ref{core-conj-prop} is the following result which generalises \cite[Proposition 3.9]{IshiharaTamaru2016}.

\begin{corollary}
If $E=G \times A$, then 
\begin{eqnarray*}
\Core(E) &\cong& \Core(G) \times \Core(A)\\
\Conj_n(E) &\cong& \Conj_n(G) \times \Conj_n(A).
\end{eqnarray*}
\end{corollary}

\begin{proposition}\label{Alex-prop}
Let $1 \to A \to E \to G \to 1$ be an extension of groups and $f \in \Aut(E)$ such that $f(A)=A$. If $f_1\in \Aut(G)$ and $f_2 \in \Aut(A)$ are the automorphisms induced by $f$,  then $\Alex_f(E) \cong \Alex_{f_1}(G) \times_\alpha \Alex_{f_2}(A)$ for some dynamical 2-cocycle $\alpha$. 
\end{proposition}

\begin{proof}
Let $\pi: E \to G$, $\kappa:G \to E$ and $\mu: G \times G \to A$ be as in Proposition \ref{core-conj-prop}, that is, $\kappa(x)*\kappa(y)= \kappa(x*y)~\mu(x, y)$ for all $x, y \in G$. Note that $\Alex(\pi): \Alex_f(E) \to \Alex_{f_1}(G)$ is a surjective quandle homomorphism. For each $x \in G$, define a bijection $\gamma_x: \Alex(\pi)^{-1}(x) \to A$ by $\gamma_x \big(\kappa(x)s\big)=s$.  Again, by \cite[Corollary 2.5]{Andruskiewitsch2003}, we obtain
$$ \Alex_f(E) \cong \Alex_{f_1}(G) \times_\alpha A,$$ where $\alpha$ is the dynamical 2-cocycle given by $\alpha_{x, y}(s, t)= \gamma_{x*y} \big(\gamma_x^{-1}(s)* \gamma_y^{-1}(t) \big)$ for $x, y \in G$ and $s, t \in A$. We check that
\begin{eqnarray*}
\alpha_{x, y}(s, t) &=& \gamma_{x*y} \big( \big(\kappa(x)s \big)* \big(\kappa(y)t \big)\big)\\
&=& \gamma_{x*y} \Big(f \big( \kappa(x)s t^{-1}\kappa(y)^{-1} \big)\big(\kappa(y)t \big) \Big)\\
&=& \gamma_{x*y} \Big(f \big(\kappa(x)\kappa(y)^{-1}\big)\kappa(y)~^{\kappa(y)^{-1}f(\kappa(y))}{f(st^{-1})}~t\Big)\\
&=&\gamma_{x*y} \Big(\big(\kappa(x)*\kappa(y)\big)~^{\kappa(y)^{-1}f(\kappa(y))}{f(st^{-1})}~t\Big)\\
&=& \gamma_{x*y} \Big( \kappa(x*y)~\mu(x, y)~^{\kappa(y)^{-1}f(\kappa(y))}{f(st^{-1})}~t\Big)\\
&=&  \mu(x, y)~^{\kappa(y)^{-1}f(\kappa(y))}{f_2(st^{-1})}~t.
\end{eqnarray*}
Taking $x=y=1$ and using the facts that $\mu(1, 1)=1=\kappa(1)$, we get $$(1, s)*(1, t)=\big(1, \alpha_{1, 1}(s, t)\big)=\big(1, f_2(st^{-1})t\big)=(1, s*t),$$ which is the desired quandle structure on $\Alex_{f_2}(A)$.
\end{proof}

Proposition \ref{Alex-prop} recovers the following special case of \cite[Proposition 3.7]{IshiharaTamaru2016}.

\begin{corollary}
If $E=G \times A$ and $f=(f_1, f_2) \in \Aut(G) \times \Aut(A)$, then 
$$\Alex_f(E) \cong \Alex_{f_1}(G) \times \Alex_{f_2}(A).$$
\end{corollary}
\medskip

In the reverse direction, it is interesting to know whether an extension of quandles $E= X \times_\alpha S$ gives rise to an extension of adjoint groups. Let $\pi : E \to X$ be the surjective quandle homomorphism given by $\pi(x, s)=s$. Then, by functoriality of adjoint groups, there is a surjective group homomorphism $$\Adj_w(\pi): \Adj_w(E) \to \Adj_w(X)$$ given by $ \Adj_w(\pi)(e_{(x, s)})=e_x$. Note that $\langle e_{(x,s)}^{-1} e_{(x,t)}~|~x \in X,~s, t \in S \rangle \le  \Ker \big(\Adj_w(\pi)\big)$. Understanding the structure of $\Ker \big(\Adj_w(\pi)\big)$ seems difficult in general. 

\begin{example}
Let $\R_4 = \{ a_0, a_1, a_2, a_3 \}$ be the dihedral quandle of order four. If $X = \{ 0, 1 \}$ is the two element trivial quandle, then there is a surjective quandle homomorphism $\pi: \R_4 \to X$ given by $\pi(a_0)=\pi(a_2)=0$ and $\pi(a_1)=\pi(a_3)=1$. Then, by \cite[Corollary 2.5]{Andruskiewitsch2003}, $$\R_4 \cong X \times_\alpha S,$$ where $\alpha$ is a dynamical 2-cocycle  and $S$ is a set of two elements. Since both $X$ and $S$ are two element trivial quandles, $\Adj(X)=\Adj(S)= \mathbb{Z}^2$, the free abelian group of rank two.  We claim that $\Adj(\R_4)$ is not an extension of $\mathbb{Z}^2$ by $\mathbb{Z}^2$.
\par
We write $\R_4$ as a disjoint union of two fibers
$$
\R_4 = Y_0 \sqcup Y_1,~\mbox{where}~Y_0 = \{ a_0, a_2 \},~~~Y_1 = \{ a_1, a_3 \}.
$$
In fact, the fibers $Y_1, Y_2$ are the connected components and trivial subquandles of $\R_4$. We see that $S_{a_0} = S_{a_2}$ with
$$
S_{a_0}(a_1) = a_3,~~S_{a_0}(a_3) = a_1.
$$
Similarly, $S_{a_1} = S_{a_3}$ with
$$
S_{a_1}(a_0) = a_2,~~S_{a_1}(a_2) = a_0.
$$
Since  $Y_0,Y_1$ and $X$ are trivial quandles, we have
$$
\Adj(Y_0) = \big\langle e_0, e_2~|~e_0 e_2 = e_2 e_0 \big\rangle,~~\Adj(Y_1) = \big\langle e_1, e_3~|~e_1 e_3 = e_3 e_3 \big\rangle
$$
and $$\Adj(X) = \big\langle f_0, f_1~|~f_0 f_1 = f_1 f_0 \big\rangle.$$
We see that the group $\Adj(\R_4)$ is generated by its subgroups $\Adj(Y_0), \Adj(Y_1)$ and has relations
$$
e_1^{-1} e_0 e_1 = e_2,~~~e_1^{-1} e_2 e_1 = e_0,~~~ e_0^{-1} e_1 e_0 = e_3,~~~e_2^{-1} e_1 e_2 = e_3,~~~ e_3^{-1} e_0 e_3 = e_2,~~~e_3^{-1} e_2 e_3 = e_0.
$$

The elements $b_0:=e_0 e_2^{-1}$ and  $b_1:=e_1 e_3^{-1}$ lie in the kernel of the homomorphism
$$
\Adj(\pi) : \Adj(\R_4) \to \Adj(X).
$$
We remove the generators $e_2$ and $e_3$ using the equalities $e_2 = e_0 b_0^{-1}$ and $e_3 = e_1 b_1^{-1}$. Then the other relations have the form
$$
b_0 = e_1^{-1} e_0^{-1} e_1 e_0,~~b_1 = e_0^{-1} e_1^{-1} e_0e_1,~~b_1 = b_0^{-1},~~b_0^2 = 1,~~e_0^{-1}b_0e_0 = b_0,~~e_1^{-1}b_1e_1 = b_1.
$$
Thus, $\Adj(\R_4)$ has a presentation
$$
\Adj(\R_4) = \big\langle e_0, e_1, b_0~|~b_0^2 = 1,~~e_1 e_0 = e_0 e_1 b_0,~~e_1^{-1} e_0 = e_0 e_1^{-1} b_0,~~ [e_0, b_0] =   [e_1, b_0] = 1 \big\rangle
$$
and there is an exact sequence
$$
1 \to \mathbb{Z}_2 \to \Adj(\R_4) \to  \mathbb{Z}^2 \to 1,
$$
where $\mathbb{Z}_2 = \langle b_0 \rangle$ and $\mathbb{Z}^2= \langle f_0, f_1 \rangle$.\end{example}
\medskip

\section{Homomorphism from group cohomology to quandle cohomology}\label{Homomorphism group cohomology to quandle cohomology}

We briefly recall the idea of a quandle module due to \cite{Andruskiewitsch2003, Jackson2005}. A {\it trunk} $\T(X)$ of a quandle $X$ is an object analogous to a category which has one object for each $x \in X$ and for each ordered pair $(x,y)\in X \times X$  there are morphisms $\textbf{a}_{x,y}: x \to x*y$ and $\textbf{b}_{y,x}: y \to x*y$ such that the diagrams
$$
\xymatrix{
x \ar[r]^{\textbf{a}_{x,y}}  \ar[d]_{\textbf{a}_{x,z}} & x*y \ar[d]^{\textbf{a}_{x*y,z}} \\
x*z \ar[r]^{\textbf{a}_{x*z,y*z}\hspace*{20mm} } & (x*y)*z=(x*z)*(y*z)}
$$
and 
$$
\xymatrix{
y \ar[r]^{\textbf{b}_{y,x}}  \ar[d]_{\textbf{a}_{y,z}} & x*y \ar[d]^{\textbf{a}_{x*y,z}} \\
y*z \ar[r]^{\textbf{b}_{y*z, x*z}\hspace*{20mm} } & (x*y)*z=(x*z)*(y*z)}
$$
commute  for all $x,y,z \in X$. If $\mathcal{A}$ is the category of abelian groups, then  a {\it trunk map} is a functor $\T(X) \to \mathcal{A}$. Thus, a trunk map assigns an abelian groups $A_x$ to each $x \in X$ and group homomorphisms
$\bar{\textbf{a}}_{x,y}:A_x \to A_{x*y}$ and $\bar{\textbf{b}}_{y,x}: A_y \to A_{x*y}$
such that
\begin{equation}\label{homo-quandle-module-1}
\bar{\textbf{a}}_{x*y,z} ~\bar{\textbf{a}}_{x,y} =  \bar{\textbf{a}}_{x*z,y*z} ~\bar{\textbf{a}}_{x,z}
\end{equation}
and
\begin{equation}\label{homo-quandle-module-2}
\bar{\textbf{a}}_{x*y,z}~ \bar{\textbf{b}}_{y,x}  =  \bar{\textbf{b}}_{y*z,x*z} ~\bar{\textbf{a}}_{y,z}
\end{equation}
for all $x,y,z \in X$. A {\it quandle module} over $X$ is a trunk map
$\T(X) \to \mathcal{A}$ such that each
$\bar{\textbf{a}}_{x,y}$ is an isomorphism and the identities
\begin{equation}\label{homo-quandle-module-3}
\bar{\textbf{b}}_{z,x*y}(s) =  \big(\bar{\textbf{a}}_{x*z,y*z}\bar{\textbf{b}}_{z,x}(s)\big)~ \big(\bar{\textbf{b}}_{y*z,x*z}\bar{\textbf{b}}_{z,y}(s)\big)
\end{equation}
and 
\begin{equation}\label{homo-quandle-module-4}
\bar{\textbf{a}}_{z, z}(s)~\bar{\textbf{b}}_{z, z}(s)=s
\end{equation}
hold for all $x,y,z \in X$ and $s \in A_z$. A quandle module over $X$ is called {\it homogeneous} if $A_x=A$ for all $x \in X$. A homogeneous quandle module is called {\it trivial} if $\bar{\textbf{a}}_{x,y}$ is the identity map and $\bar{\textbf{b}}_{x,y}$ is the trivial map for all $x, y \in X$.
\par

\begin{example}\label{homogen-module-example}
Any abelian group $A$ can be turned into a homogeneous quandle module over a quandle $X$ by taking $\bar{\textbf{a}}_{x,y}(s)=s^{-1}$ and $\bar{\textbf{b}}_{x,y}(s)=s^2$ for all $x, y \in X$ and $s \in A$.
\end{example}

From now onwards, we would consider only homogeneous quandle modules over a quandle. If $A$ is a homogeneous quandle module over $X$, then a {\it factor set} is a map $\mu: X \times X \to A$ satisfying
\begin{equation}\label{factor-set-def-condition}
\mu(x*y,z) ~ \bar{\textbf{a}}_{x*y,z}\big(\mu(x,y)\big) = \bar{\textbf{a}}_{x*z,y*z}\big(\mu(x,z)\big)  ~\mu(x*z,y*z)~ \bar{\textbf{b}}_{y*z,x*z}\big(\mu(y,z)\big)
\end{equation}
and 
\begin{equation}\label{factor-set-def-condition-2}
\mu(x,x) = 1
\end{equation}
for all $x,y,z \in X$. Identities \eqref{homo-quandle-module-1}-\eqref{homo-quandle-module-4} imply that each factor set $\mu$ defines a quandle extension $E=X \times_\alpha A$ of $X$ by $A$ by
\begin{equation}\label{factor-set-quandle-operation}
(x, s)* (y, t)= \big(x*y, ~\bar{\textbf{a}}_{x,y}(s)~\mu(x, y)~\bar{\textbf{b}}_{y,x}(t) \big)
\end{equation}
for all $x, y \in X$ and $s, t \in A$. Conversely, by \cite[Proposition 3.1]{Jackson2005}, every extension of a quandle $X$ by a homogeneous quandle module $A$ over $X$ is determined by a factor set. Further, two factor sets $\mu_1,\mu_2: X \times X \to A$ are {\it cohomologous} if there exists a map $\lambda:X \to A$ such that
\begin{equation}\label{factor-set-equivalence}
\mu_2(x, y) = \mu_1(x, y)~\bar{\textbf{a}}_{x,y} \big(\lambda(x) \big)~\bar{\textbf{b}}_{y,x} \big(\lambda(y) \big)~\lambda(x*y)^{-1}
\end{equation}
 for all $x,y \in X$. By \cite[Theorem 3.6]{Jackson2005}, two factor sets are cohomologous if and only if the corresponding extensions are equivalent. Thus, the set $\mathcal{H}^2(X;A)$ of cohomology classes of factor sets, which forms an abelian group with point-wise multiplication,  is in bijective correspondence with the set of equivalence classes of quandle extensions of $X$ by $A$. When $A$ is a trivial homogeneous quandle module over $X$, then $\mathcal{H}^2(X;A)=\Ho^2(X;A)$ and the quandle operation \eqref{factor-set-quandle-operation} agrees with \eqref{abelian-extension-operation}.
\par

Next, we recall a bit of extension theory of groups that we need for our last two results. Let $A$ be an abelian group and $G$ a group acting on $A$ from left by automorphisms. The second cohomology of $G$ with coefficients in $A$ is defined as $\Ho^2(G; A):=\Z^2(G; A)/\B^2(G; A)$, where
$$\Z^2(G; A) =\big\{\nu:G \times G \to A~|~ ^{x}\nu(y,z)\nu(x, yz)= \nu(xy, z)\nu(x, y)~ \textrm{for all}\ x, y, z\in G\big\}$$
is the group of  {\it 2-cocycles} and
$$\B^2(G;A) = \big\{\nu:G \times G \to A~|~\textrm{there exists}~\lambda:G \to A~\textrm{with}~ \nu(x, y)=~^{x}\lambda(y) \lambda(xy)^{-1} \lambda(x)~\textrm{for all}~x, y\in G \big\}$$
is the group of {\it 2-coboundaries}.
\medskip

It is well-known that $\Ho^2(G;A)$ classifies extensions of groups $1 \to A \to E \to G \to 1$ that give rise to the given action of $G$ on $A$. In particular, $E$ is abelian if and only if any group 2-cocycle $\nu$ classifying the extension is symmetric, that is, 
$$\nu(x, y)=\nu(y, x)$$ for all $x, y \in G$. Since any group 2-cocycle cohomologous to a symmetric group 2-cocycle is itself symmetric, we can consider the set $\Ho^2(G;A)_{sym}$ of all symmetric group cohomology classes of degree 2. In fact, it turns out that $\Ho^2(G;A)_{sym}$ is a subgroup of $\Ho^2(G;A)$.

\begin{theorem}\label{homo-group-coho-quandle-coho}
Let $G$ and $A$ be abelian groups with $A$ viewed as a trivial $G$-module. Then there is a natural group homomorphism $\Lambda:\Ho^2(G;A)_{sym} \to \mathcal{H}^2 \big(\Core(G);A\big )$ given by $\Lambda \big([\nu] \big) = [\check{\nu}]$.
\end{theorem}

\begin{proof}
Let $1 \to A \to E \to G \to 1$ be an extension of abelian groups and $\kappa:G \to E$ a transversal such that $\kappa(1)=1$. Let  $\nu: G \times G \to A$ be a symmetric group 2-cocycle corresponding to the  transversal $\kappa$, that is, 
\begin{equation}\label{tau-nu}
\kappa(x)\kappa(y)= \kappa(xy)~\nu(x, y)
\end{equation}
for all $x, y \in G$. Define $\check{\nu}: G \times G \to A$ by 
$$\kappa(x)*\kappa(y)=\kappa(x*y)~\check{\nu}(x, y)$$
for all $x, y \in G$. Then, we have
\begin{eqnarray}\label{mu-nu-relation}
\nonumber \check{\nu}(x, y) &=& \big( \kappa(x)*\kappa(y) \big)\kappa(x*y)^{-1}\\
\nonumber &=& \kappa(y)\kappa(x)^{-1} \kappa(y)\kappa(yx^{-1}y)^{-1}\\
\nonumber  &=& \kappa(y)\kappa(x)^{-1} \kappa(yx^{-1})^{-1} ~\nu(yx^{-1}, y),~\textrm{using \eqref{tau-nu} and the fact that $E$ is abelian}\\
 &=& \nu(yx^{-1}, x)^{-1}~ \nu(yx^{-1}, y).
\end{eqnarray}
As in Example \ref{homogen-module-example}, we view $A$ as a homogeneous quandle module over $\Core(G)$ by taking $\bar{\textbf{a}}_{x,y}(s)=s^{-1}$ and $\bar{\textbf{b}}_{x,y}(s)=s^2$ for all $x, y \in \Core(G)$ and $s \in A$. By Proposition \ref{core-conj-prop}, $\Core(E)$ is a quandle extension of $\Core(G)$ by the homogeneous quandle module $A$. Further, $\check{\nu}$ is the factor set of the extension $\Core(E)$ relative to the transversal $\kappa$ in the sense of \eqref{factor-set-def-condition}. 
\par

We claim that if $\nu_1, \nu_2: G \times G \to A$ are two cohomologous  group  2-cocycles, then the corresponding factor sets $\check{\nu_1}, \check{\nu_2}$ are cohomologous in the sense of \eqref{factor-set-equivalence}. Let $\lambda:G \to A$ be a map such that $\nu_2(x, y)=\nu_1(x,y) \lambda(x)\lambda(xy)^{-1}\lambda(y)$ for all $x, y \in G$. Then, using \eqref{mu-nu-relation}, we obtain 
\begin{eqnarray*}
\check{\nu_2}(x, y) &=& \nu_2(yx^{-1}, x)^{-1} \nu_2(yx^{-1}, y)\\
 &=& \big(\nu_1(yx^{-1}, x)^{-1}~  \nu_1(yx^{-1},y) \big) \lambda(y)^2 \lambda(x)^{-1} \lambda(yx^{-1}y )^{-1}\\
 &=& \check{\nu_1}(x, y) \lambda(y)^2 \lambda(x)^{-1} \lambda(x*y )^{-1}\\
 &=& \check{\nu_1}(x, y) ~\bar{\textbf{b}}_{y, x}\big( \lambda(y) \big)~\bar{\textbf{a}}_{x,y} \big(\lambda(x)\big) ~\lambda(x*y )^{-1}
\end{eqnarray*}
for all $x, y \in G$. Thus, by \eqref{factor-set-equivalence}, $\check{\nu_1}$ and $\check{\nu_2}$ are cohomologous and determine equivalent extensions. This defines a map $$\Lambda:\Ho^2(G;A)_{sym} \to \mathcal{H}^2 \big(\Core(G);A\big)$$ given by $\Lambda\big([\nu] \big)= [\check{\nu}]$. If $[\nu_1], [\nu_2] \in 
\Ho^2(G;A)_{sym}$, then 
\begin{eqnarray*}
\Lambda\big([\nu_1][\nu_2] \big) &=& \Lambda\big([\nu_1\nu_2] \big) \\
&=& [\check{(\nu_1\nu_2)}]\\
&=& [\check{\nu_1}\check{\nu_2}]\\
&=& [\check{\nu_1}][\check{\nu_2}]\\
&=& \Lambda\big([\nu_1]\big) \Lambda\big( [\nu_2] \big),
\end{eqnarray*}
and hence $\Lambda$ is a group homomorphism. 
\par

For naturality of  $\Lambda$, let $G_1, G_2, A_1, A_2$ be abelian groups and $f:G_2 \to G_1$, $h:A_1 \to A_2$ be group homomorphisms. Then for each group 2-cocycle $\nu:G_1 \times G_1 \to A_1$, we define
$$\nu'(x, y)=h \big(\nu\big(f(x), f(y)\big) \big)$$ for all $x, y \in G_2$. A routine check shows that $\nu \mapsto \nu'$ induces a group homomorphism $\Ho^2(G_1;A_1)_{sym} \to \Ho^2(G_2;A_2)_{sym}$ given by $[\nu] \mapsto [\nu']$. Similarly, for each factor set $\mu: \Core(G_1) \times \Core(G_1) \to A_1$, define 
$$\mu'(x, y)=h \big(\mu\big(f(x), f(y)\big) \big)$$ for all $x, y \in \Core(G_2)$. A direct check shows that $\mu'$ satisfies the factor set conditions \eqref{factor-set-def-condition} and \eqref{factor-set-def-condition-2}. Since cohomologous factor sets are mapped to cohomologous factor sets, we obtain a group homomorphism  $\mathcal{H}^2 \big(\Core(G_1); A_1\big) \to \mathcal{H}^2 \big(\Core(G_2); A_2\big)$ given by $[\mu] \mapsto [\mu']$. Finally, it follows from the construction of the preceding maps that the diagram 
$$
\xymatrix{
\Ho^2(G_1;A_1)_{sym} \ar[r]^{[\nu] \mapsto [\nu']} \ar[d]^{\Lambda} & \Ho^2(G_2;A_2)_{sym} \ar[d]^{\Lambda} \\
\mathcal{H}^2 \big(\Core(G_1); A_1\big) \ar[r]^{[\mu] \mapsto [\mu']}  & \mathcal{H}^2 \big(\Core(G_2); A_2\big)}
$$
commutes, and the map $\Lambda$ is natural.
\end{proof}

Let $G$ be a group and $A$ a trivial $G$-module. For a group 2-cocycle $\nu \in \Z^2(G;A)$, define $\breve{\nu}:G \times G \to A$  by 
\begin{equation}\label{nu-nu-breve-relation}
\breve{\nu}(x, y)=\nu(x, y)~ \nu(y, y^{-1}xy)^{-1}
\end{equation}
 for $x, y \in G$. It is shown in \cite[Theorem 7.1]{CJKLS2003} that $\breve{\nu} \in \Z^2\big( \Conj(G);A\big)$.  In fact, the following analogue of Theorem \ref{homo-group-coho-quandle-coho} holds.

\begin{theorem}\label{group-homo-grouo-coho-conj}
Let $G$ be a group and $A$ an abelian group viewed as a trivial $G$-module. Then there is a natural group homomorphism $\Gamma:\Ho^2(G;A) \to \Ho^2\big(\Conj(G),A\big)$ given by $\Gamma \big([\nu] \big) = [\breve{\nu}]$.
\end{theorem}

\begin{proof}
Let $\nu_1, \nu_2$ be two cohomologous  group  2-cocycles. Then there exists a map $\lambda:G \to A$ such that $\nu_2(x, y)=\nu_1(x,y) \lambda(x)\lambda(xy)^{-1}\lambda(y)$ for all $x, y \in G$. Using \eqref{nu-nu-breve-relation}, we obtain 
\begin{eqnarray*}
\breve{\nu_2}(x, y) &=& \nu_2(x, y) ~\nu_2(y, y^{-1}xy)^{-1}\\
 &=& \big(\nu_1(x,y) ~\nu_1(y, y^{-1}xy)^{-1} \big)\lambda(x)\lambda(y^{-1}xy)^{-1}\\
 &=& \breve{\nu_1}(x, y)\lambda(x)\lambda(x*y)^{-1}.
\end{eqnarray*}
for all $x, y \in G$. Thus, $\breve{\nu_1}$ and $\breve{\nu_2}$ are cohomologous quandle 2-cocycles. This defines a map $$\Gamma:\Ho^2(G;A) \to \Ho^2 \big(\Conj(G),A\big)$$ given by $\Gamma\big([\nu] \big)= [\breve{\nu}]$. If $[\nu_1], [\nu_2] \in 
\Ho^2(G;A)$, then 
\begin{eqnarray*}
\Gamma\big([\nu_1][\nu_2] \big) &=& \Gamma\big([\nu_1\nu_2] \big) \\
&=& [\breve{(\nu_1\nu_2)}]\\
&=& [\breve{\nu_1}\breve{\nu_2}]\\
&=& [\breve{\nu_1}][\breve{\nu_2}]\\
&=& \Gamma\big([\nu_1]\big) \Gamma\big( [\nu_2] \big),
\end{eqnarray*}
and $\Gamma$ is a group homomorphism. Naturality of $\Gamma$ follows similar to that of $\Lambda$ in Theorem \ref{homo-group-coho-quandle-coho}.
\end{proof}

We conclude with the following  natural problem. 

\begin{problem}
Determine  the kernel and the image of the homomorphisms $\Lambda$ and $\Gamma$.
\end{problem}
\medskip

\noindent\textbf{Acknowledgments.}
Bardakov is supported by the Russian Science Foundation (project 19-41-02005) for work in sections 3 and 4 and by the Ministry of Science and Higher Education of Russia (agreement No. 075-02-2020-1479/1) for work in sections 8 and 9. The work of Singh is supported by the Swarna Jayanti Fellowship grants DST/SJF/MSA-02/2018-19 and SB/SJF/2019-20/04.

\end{document}